\documentclass[reqno, 12pt]{amsart}
\usepackage{epsfig,amsmath,amsfonts,latexsym}
\usepackage{amsthm}
\usepackage[abs]{overpic}
\usepackage[usenames,dvipsnames]{color}
\usepackage{palatino}
\usepackage[hyperfootnotes=false]{hyperref}
\usepackage{caption}
\usepackage{dsfont, url}
\usepackage{mathtools}
\hypersetup{
  colorlinks,
  citecolor=Purple,
  linkcolor=Black,
  urlcolor=Purple}
  
\theoremstyle{plain}

\newtheorem{theorem}{Theorem}

\newtheorem*{Conventions}{Conventions}
\newtheorem{dummy}{anything}[section]
\newtheorem*{thm}{Theorem}
\newtheorem{lemma}[dummy]{Lemma}

\theoremstyle{definition}
\newtheorem{definition}[dummy]{Definition}

\newtheorem{remark}[dummy]{Remark}

\theoremstyle{remark}
\newcommand{\op}{\operatorname}

\newcommand{\R}{\mathbb{R}}

\def\R{\mathbb{R}}

\begin{document}

\title[Planarity in higher-dimensional contact manifolds]{Planarity in higher-dimensional contact manifolds}

\author{Bahar Acu}

\address[B.\ Acu]{Department of Mathematics \\ Northwestern University \\ Evanston \\ IL \\ USA}

\email{baharacu@northwestern.edu}
\email{baharacu@gmail.com}

\author{Agustin Moreno}

\address[A.\ Moreno]{Institut f\"{u}r Mathematik \\ Universit\"at Augsburg \\  Augsburg  \\ Germany}

\email{agustin.moreno2191@gmail.com}




\begin{abstract} We obtain several results for (iterated) planar contact manifolds in higher dimensions: (1) Iterated planar contact manifolds are not weakly symplectically co-fillable. This generalizes a $3$-dimensional result of Etnyre \cite{Etn} to a higher-dimensional setting, where the notion of weak fillability is that due to Massot-Niederkr\"uger-Wendl \cite{MNW}. (2) They do not arise as nonseparating weak contact-type hypersurfaces in closed symplectic manifolds. This generalizes a result by Albers-Bramham-Wendl \cite{ABW}. (3) They satisfy the Weinstein conjecture, i.e. every contact form admits a closed Reeb orbit. This is proved by an alternative approach as that of \cite{Acu}, and is a higher-dimensional generalization of a result of Abbas-Cieliebak-Hofer \cite{ACH}. The results follow as applications from a suitable symplectic handle attachment, which bears some independent interest.
\end{abstract}

\maketitle

\section{Introduction}

Contact structures in every odd dimension can be understood via open book decompositions, a correspondence which has been established by the important work of Giroux \cite{Gir}. In dimension three, planar contact manifolds, those that correspond to an open book decomposition with genus zero pages, are in some sense the simplest contact 3-manifolds. For instance, their strong symplectic fillings, when they exist, carry Lefschetz fibrations inducing the given planar open book decomposition along their boundary \cite{Wen}. This implies that the classification of strong symplectic fillings of a planar contact manifold boils down to studying Dehn twist factorizations in the mapping class of a genus zero surface. \smallskip

While overtwisted contact structures are planar \cite{Etn}, there are known obstructions to planarity in dimension three. For instance, Etnyre \cite{Etn} proves that if $W$ is a weak symplectic filling of a contact 3-manifold $M$ supported by a planar open book decomposition, then the boundary of $W$ is connected. Moreover, planar contact 3-manifolds do not admit nonseparating embeddings as weak contact-type hypersurfaces inside closed symplectic 4-manifolds \cite{ABW, NW}. Etnyre's result can actually be recovered from this result (in the strong filling case), as it follows easily from the existence of symplectic caps in dimension three \cite{Eli,Etn2}. \smallskip

In higher dimensions, there are natural generalizations of planarity. For instance, one can replace a fibration defining an open book decomposition by a fibration whose fibers are still genus zero surfaces with boundary, but such that the base is now a higher-dimensional closed contact manifold. One could also consider contact manifolds admitting a submanifold with the structure of a  contact fibration over a Liouville domain of arbitrary dimension, whose fibers are closed planar contact $3$-manifolds. Both of these lead naturally to (planar versions of) the notion of \emph{spinal open book decomposition} or SOBD, as discussed in \cite{Mo} based on the three dimensional notion defined in \cite{SOBD}. In this paper, however, we will take an alternative point of view. \smallskip

An alternative generalization is to consider standard open book decompositions in higher dimensions, keeping $S^1$ as the base of the symplectic fibration, but impose the condition that the fibers carry a suitable structure which is inductively built from a low-dimensional planar structure. In this vein, the following notion was introduced in \cite{Acu}. \smallskip

Given a $(2n-2)$-dimensional Weinstein domain $(W^{2n-2}, \omega)$, we say that $W$ admits an \textit{iterated planar Lefschetz fibration} if there exists a sequence $f_i: W^{2i} \to \mathbb{D}^2$ of exact symplectic Lefschetz fibrations, $i=2,\dots,n-1$, where the regular fiber of $f_{i+1}$ is the total space of $f_{i}$, and $f_2: W^4 \to \mathbb{D}^2$ is a planar Lefschetz fibration. We denote its regular fiber by $W^2$, which is simply a genus zero surface with boundary. Observe that when $n=3$, an iterated planar Lefschetz fibration is a planar Lefschetz fibration. We point out that the notion of an iterated Lefschetz fibration (not necessarily planar) has also appeared in \cite{Seidel}, under the name of \emph{bifibration}.\smallskip

An \textit{iterated planar contact manifold} $(M, \xi)$ is a $(2n-1)$-dimensional contact manifold supported by an open book decomposition $\text{OB}(W^{2n-2}, \varphi)$ whose page $W^{2n-2}$ admits an iterated planar Lefschetz fibration. Note that when $n=3$, $M$ is a contact 5-manifold supported by an open book decomposition whose page admits a planar Lefschetz fibration, inducing a planar open book decomposition in the binding. We refer the reader to Section \ref{preliminaries} and also \cite{AEO} for a more detailed discussion on iterated planarity. \smallskip

The notion of weak symplectic fillability, as introduced in \cite{MNW} in arbitrary dimensions, will be relevant for the statement of the results in this paper. A contact manifold $(M^{2n-1},\xi)$ is \emph{weakly fillable} if it is the smooth boundary of a symplectic $2n$-manifold $(W^{2n},\omega)$, such that $(\omega+\tau d\lambda)\vert_{\xi}$ is symplectic for every $\tau\geq 0$, for one (and hence every) choice of contact form $\lambda$. We refer to the latter condition as \emph{weak domination}. One then says that $(M,\xi)$ is \textit{weakly dominated} by $\omega$, and that $(W,\omega)$ is a weak filling for $(M,\xi)$. On can similarly define the notion of a \emph{weak contact-type embedding}. Namely, if the contact manifold $(M,\xi)$ is given, one says that it embeds as a weak contact-type hypersurface in the symplectic manifold $(W,\omega)$, if there exists an embedding $M \hookrightarrow W$ satisfying the weak domination condition given above. See Section \ref{preliminaries} for further relevant background.

\begin{Conventions} Throughout the paper, we will abbreviate \textit{iterated planar} as \textbf{IP}, we will refer to an open book decomposition as in the above definition as an \textbf{IP} open book decomposition, and we will denote $M^{2n-1}=\textbf{OB}(W^{2n-2},\varphi)$. Lastly, when we talk about a weak symplectic co-filling, we will mean a symplectic manifold with disconnected weakly dominated contact boundary. Each contact manifold arising as a boundary component of a (weak) co-filling is said to be (weakly) co-fillable. We will usually keep track of dimensions in the notation, in order to avoid confusion (while perhaps making the notation a bit cumbersome in some places).
\end{Conventions}

\subsection*{Statements of the results} In \cite{MNW}, several examples of (exactly) symplectically co-fillable higher-dimensional contact manifolds were constructed. It is then natural to wonder which contact manifolds can fit into a symplectic co-filling. The first result of this paper, generalizing Etnyre's $3$-dimensional result \cite{Etn}, is an obstruction to iterated planarity. The statement reads as follows:

\begin{theorem}\label{mainthm} \textbf{IP} contact manifolds are not weakly symplectically co-fillable.
\end{theorem}

In other words, if a contact manifold $M$ is weakly symplectically co-fillable, then it cannot be iterated planar. A further obstruction to iterated planarity, which generalizes a result by Albers-Bramham-Wendl \cite{ABW}, is the following: 

\begin{theorem}\label{nonsep} \textbf{IP} contact manifolds do not embed as nonseparating weak contact-type hypersurfaces in closed symplectic manifolds. 
\end{theorem}

In other words, if an \textbf{IP} contact manifold admits a weak contact-type embedding into a closed symplectic manifold, then it separates the latter into two disjoint pieces.\smallskip

By the work of Lisca and Mati\'{c} \cite{LM}, we know that Stein fillings of contact manifolds can be embedded into closed symplectic manifolds. That is, Stein fillable contact manifolds admit symplectic caps. For general higher-dimensional contact manifolds, it was not known until very recently that they admit (strong) symplectic caps \cite{CE, Laz}. This fact can be used to recover Theorem \ref{mainthm} (with the word ``weakly'' replaced by ``strongly'') from Theorem \ref{nonsep}, but our proof is independent of the existence of caps. \smallskip

Recall that given a contact form $\lambda$ for a contact manifold $(M, \xi)$, there exists a unique contact vector field called the \textit{Reeb vector field} $R_{\lambda}$, defined by the equations $\iota_{R_{\lambda}}d\lambda=0$ and $\lambda(R_{\lambda})=1$. The Weinstein conjecture \cite{Wei} states that, on a compact contact manifold $(M,\xi)$ of any odd dimension, any Reeb vector field $R_{\lambda}$ for $\xi$ carries at least one closed periodic orbit. This is a central conjecture in contact and symplectic topology which has been inspirational for many developments in these fields and has a rich history. In dimension three, it was established by Taubes \cite{Tau} (based on Seiberg-Witten theory), which culminated a large body of work by several people extending over more than two decades. In higher dimensions, though there are several partial results \cite{AH,FHV,HV,Vit}, it is still open. \smallskip

Prior to Taubes's work, Abbas-Cieliebak-Hofer \cite{ACH} developed a program for proving the Weinstein conjecture based on open book decompositions, and proved the conjecture for planar contact $3$-manifolds. In this paper, we prove the Weinstein conjecture for the case of \textbf{IP} contact manifolds, by an alternative approach as that of \cite{Acu}, which is a generalization of the result by Abbas-Cieliebak-Hofer to higher dimensions. The proof of the following is a suitable adaptation of the proof in \cite{AH} for overtwisted contact manifolds: 

\begin{theorem}\label{WeinsteinIP} \textbf{IP} contact manifolds satisfy the Weinstein conjecture. 
\end{theorem}

\begin{remark}
In fact, \textbf{IP} contact manifolds satisfy a strong version of the Weinstein conjecture, similarly as in \cite{ACH}. Namely, there exist finitely many Reeb orbits whose homology classes sum to zero in the first homology group. This follows directly from the proof. 
\end{remark}

\subsection{Sketch of proofs} The technical input for obtaining the results is a suitable symplectic handle attachment, which is inspired by the handle attachments in \cite{Eli} and \cite{SOBD}. Similar ideas have also been discussed by the first author and collaborators for \cite{AEO}; see also \cite{DGZ14} for a similar construction. We fix an \textbf{IP} contact manifold $M_1^{2n-1}=\textbf{OB}(W_1^{2n-2}, \varphi)$ where $W_1^{2n-2}$ admits an \textbf{IP} Lefschetz fibration, yielding an \textbf{IP} open book decomposition in $M_1^{2n-1}$. The manifold $W_1^{2n-2}$ carries a natural homotopy class of Weinstein structures, which gives a supported contact structure in $M_1^{2n-1}$ via a construction due to Giroux. We will construct a symplectic cobordism $(C^{2n},\omega_{C^{2n}})$ from $M_1^{2n-1}$ (a strongly concave boundary component), to a new manifold $M_2^{2n-1}$, which is a stable boundary component of $C^{2n}$, and so carries a stable Hamiltonian structure induced from the symplectic structure $\omega_{C^{2n}}$ in $C^{2n}$. We shall do this by induction in the dimension, constructing a ``handle'' at each step, using as input the handle constructed in the previous step of the induction. The base case corresponds to the first step of Eliashberg's capping construction \cite{Eli} (See also the \emph{spine removal surgery} in \cite{SOBD}, and cf.\ \cite{Etn2}).\smallskip

Topologically, the effect of this handle attachment is to replace the given \textbf{IP} open book decomposition in $M_1^{2n-1}$ by a \emph{topological} open book decomposition in $M_2^{2n-1}$. Its pages $W^{2n-2}_2=W^{2n-2}_1 \cup C^{2n-2}$ have been enlarged by a symplectic cobordism $C^{2n-2}$ (constructed in the previous step, where we replace $M_1^{2n-1}$ by $M_1^{2n-3}=\partial W_1^{2n-2}$), the monodromy $\varphi$ extends to $W_2^{2n-2}$ by the identity along $C^{2n-2}$, and the new binding $M_2^{2n-3}=\partial W_2^{2n-2}$ has an abundance of embedded $2$-spheres. Here, we make a distinction between a topological open book, and a contact open book, since the pages of the former type of open book are no longer Liouville, and hence the manifold $M_2^{2n-1}$ does not carry any obvious supported contact structure. We shall denote this by $M_2^{2n-1}=\textbf{TOB}(W_2^{2n-2},\varphi)$. For a suitable almost complex data on $C^{2n}$ compatible with the stable Hamiltonian structure at $M_2^{2n-1}$, and which makes $M_2^{2n-1}$ a weakly pseudoconvex boundary component, these spheres in the new binding $M_2^{2n-3}$ become holomorphic.\smallskip

After adding a marked point, the virtual dimension of the resulting moduli space is $2n$, and we obtain an evaluation map. One can show that a local uniqueness lemma holds (Lemma \ref{locuniq} below), implying that the only spheres in the moduli space that touch a suitable subregion of $M_2^{2n-1}$ are the ones that we constructed ``by hand''. Moreover, the latter can be shown to be Fredholm regular, so that the moduli space is a manifold of dimension $2n$ near them.\smallskip

The cobordism $C^{2n}$ can be further modified, so that it may be glued to any symplectic manifold having $M_1^{2n-1}$ as a weak boundary component. Whereas the stable Hamiltonian structure along $M_2^{2n-1}$ is slightly perturbed, the modification may be chosen so that it does not affect the holomorphic spheres and the Local Uniqueness Lemma. The regions along which the original stable Hamiltonian structure is actually contact, become weakly dominated after the perturbation, whereas the region where the spheres are defined, and the Local Uniqueness Lemma holds, remains stable. \smallskip

The unifying idea for all the results is inspired by the proof of Theorem 6.1 in \cite{MNW}. Given a symplectic manifold $W^{2n}$ having $M_1^{2n-1}$ as a --weak or strong-- contact-type boundary component (either a hypothetical weak symplectic co-filling, or one we construct), we attach the cobordism $(C^{2n},\omega_{C^{2n}})$ to $W^{2n}$. We thus obtain a new symplectic manifold which we still call $W^{2n}$, having stable boundary $M^{2n-1}_2$. We then extend the moduli space of spheres to $W^{2n}$ together with the evaluation map. After choosing a generic and properly embedded path in $W^{2n}$ starting from the region where local uniqueness holds, we study the sub-moduli space of spheres which are constrained to intersect this path via the evaluation map, which has virtual dimension equal to $1$. While, after taking the Gromov compactification, there could be multiply covered spheres in this moduli preventing transversality by standard methods, this is dealt with via the polyfold machinery of Hofer-Wysocki-Zehnder \cite{HWZ} (in the Gromov-Witten case). After introducing an abstract and generic multivalued perturbation to the Cauchy-Riemann equation, and appealing to the polyfold regularization of constrained moduli spaces as in \cite{Ben}, we obtain a 1-dimensional, compact, oriented, weighted branched orbifold with non-empty boundary. Moreover, we may choose the abstract perturbation away from the subregion where the spheres are known to be regular and local uniqueness holds. In the case where $W^{2n}$ is compact (e.g.\ a co-filling), this moduli space is also compact. In the case where $W^{2n}$ has a noncompact end, this moduli space fails to be compact, but its noncompactness is tractable, in the sense that it corresponds to spheres escaping down the noncompact end of $W^{2n}$.\smallskip

\begin{figure}[t]\centering
\includegraphics[width=0.5\linewidth]{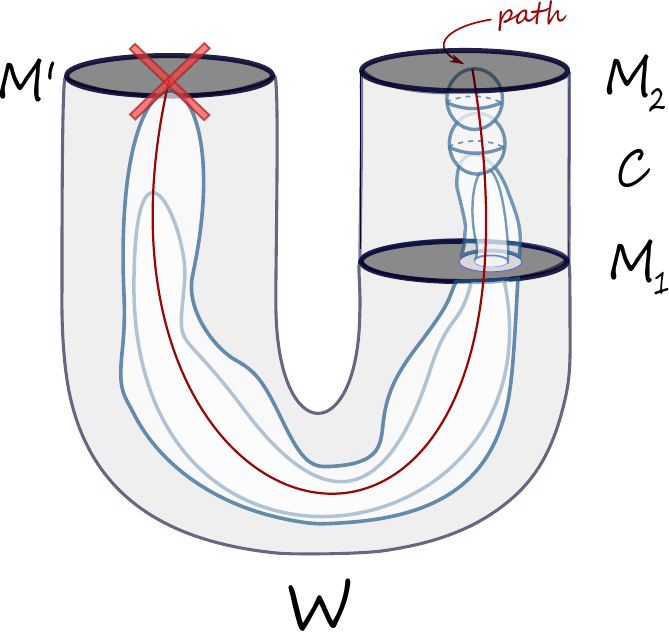}
\caption{\label{contradiction} In the proof of Theorem \ref{mainthm}, the curves of the constrained moduli space start at the boundary component of the hypothetical co-filling $W$ corresponding to the iterated planar contact manifold $M_1$. But then they hit an end corresponding to a different boundary component, which contradicts with strict pseudoconvexity.}
\end{figure}

For Theorem \ref{mainthm}, $W^{2n}$ is a hypothetical weak symplectic co-filling. The boundary of the resulting $1$-dimensional compact, oriented, weighted branched orbifold $\overline{\mathcal{M}}_{\bullet,l}$ corresponds to spheres in the constrained and perturbed moduli space which touch a boundary component of $W^{2n}$. The local uniqueness statement then yields that there is a unique curve in the boundary of this moduli space which touches the boundary component $M_2^{2n-1}$ (along the \emph{weakly} pseudoconvex subregion), so that the curves in the other boundary components of $\overline{\mathcal{M}}_{\bullet,l}$ necessarily touch other boundary components of $W^{2n}$ tangentially. But this is a contradiction, since we can take the almost complex structure so that the latter is \emph{strictly} pseudoconvex. The situation is depicted in Figure \ref{contradiction}. Observe that, while the spheres are not necessarily holomorphic (they are solutions to a perturbed Cauchy-Riemann equation), we can still appeal to pseudoconvexity for sufficiently small abstract perturbation to conclude that no perturbed spheres touch $M^\prime$ tangentially, as follows easily from SFT-compactness.\smallskip

\begin{figure}[t]\centering
\includegraphics[width=0.75\linewidth]{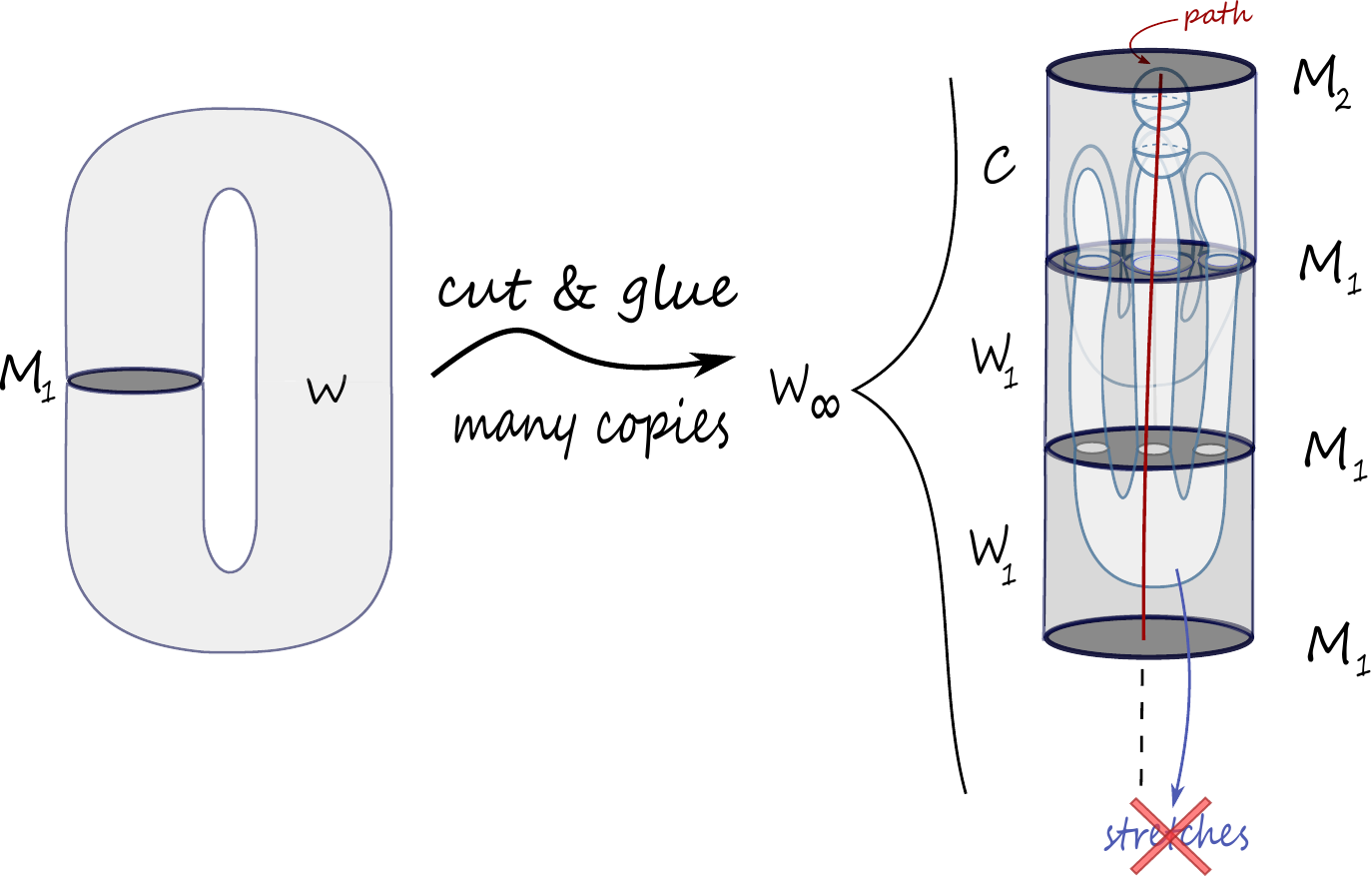}
\caption{\label{emb} In the proof of Theorem \ref{nonsep}, the curves of the constrained moduli space start at the boundary component stretch all the way down along the negative end, resulting in a contradiction.}
\end{figure}

For Theorem \ref{nonsep}, we adapt the proof in \cite{ABW}. Given a hypothetical weak contact-type embedding of $M_1^{2n-1}$ into a closed symplectic manifold $W$, we cut the latter along the former. The result is a weak symplectic cobordism $W_1^{2n}$ having $M_1^{2n-1}$ both as positive and negative weakly dominated boundary components. As in \cite{ABW}, we glue infinitely many copies of $W_1^{2n}$ to itself at the negative end, by inductively identifying positive end with negative end. The result is a noncompact symplectic manifold $W_\infty^{2n}$ with weak boundary $M_1^{2n-1}$, to which we may apply the symplectic handle attachment construction described above. The resulting (perturbed) moduli space of spheres is a $1$-dimensional non-compact, oriented, weighted branched orbifold $\overline{\mathcal{M}}_{\bullet,l}$, with one boundary component. By construction, $W_\infty^{2n}$ is periodic, and has infinitely many copies of $M_1^{2n-1}$ sitting as weak contact-type hypersurfaces which can be made strictly pseudoconvex. Periodicity implies that the geometry of $W_\infty^{2n}$ is bounded, and so closed holomorphic curves with bounded energy have bounded diameter. But then the spheres in $\overline{\mathcal{M}}_{\bullet,l}$ cannot shoot all the way down the noncompact end, since otherwise this would contradict pseudoconvexity. They also cannot stretch infinitely far down while staying at a bounded distance from the boundary of $W_\infty^{2n}$, because of the diameter bounds. This is a contradiction. See Figure \ref{emb}. \smallskip

\begin{figure}[t]\centering
\includegraphics[width=0.5\linewidth]{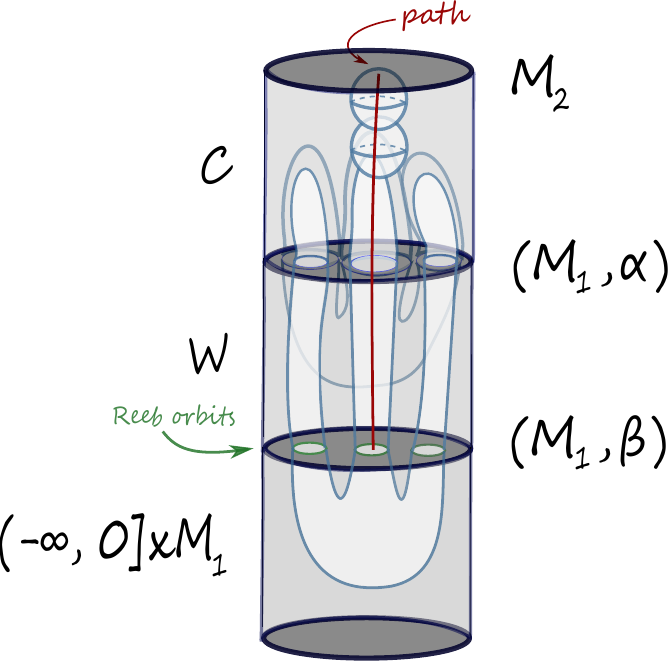}
\caption{\label{Weinsteinconj} In the proof of Theorem \ref{WeinsteinIP}, the curves of the constrained moduli space break at the negative ends, resulting in at least one closed Reeb orbit. Here, $\alpha$ is adapted to the open book, whereas $\beta$ is an arbitrary contact form inducing the given contact structure.}
\end{figure}

The arguments for Theorem \ref{WeinsteinIP} are a reformulation of those in \cite{AH}. We construct an exact symplectic cobordism from an arbitrary contact form to a specific Giroux form $\alpha$ supported by the \textbf{IP} open book decomposition (see Definition \ref{def:Girouxform}). We attach the symplectic handle, and study the resulting non-compact, oriented, weighted branched orbifold of spheres. Then, SFT-compactness \cite{SFT} and the exactness away from the handle implies breaking at the negative end, resulting in the desired closed Reeb orbit. See Figure \ref{Weinsteinconj}.

\subsection*{Acknowledgments} The authors thank Chris Wendl for suggesting these results together with simplified proofs, both during the 25th G\"okova Geometry and Topology Conference and during later conversations and correspondences. They also thank anonymous referees for many helpful suggestions and corrections resulting in an improved exposition of the paper; John Etnyre for pointing out relevant omitted citations in earlier versions of this manuscript; Otto van Koert for helpful conversations, interest in this work, and feedback on simply connected contact $5$-manifolds; Benjamin Filippenko and Zhengyi Zhou for clarifying subtleties and pointing out references regarding polyfolds and their literature.

\section{Background} \label{preliminaries}

In this section, we recall various basic facts about open book decompositions, Lefschetz fibrations, and almost complex structures, and discuss some examples. Readers who are already familiar with the basics are encouraged to skip this section and proceed to Section \ref{Handleattachment}, and later consult the current section for the relevant notation and definitions.\smallskip

Unless otherwise indicated, throughout this paper, all contact structures are positive and co-oriented. Contact and symplectic manifolds will be oriented by their contact and symplectic structures, respectively.

\subsection*{Open book decompositions} 

An \textit{open book decomposition} of an oriented $(2n+1)$-dimensional manifold $M$ is a pair $(B,\pi)$, where $B$ is a codimension $2$ submanifold of $M$ with trivial normal bundle, and $\pi: M\backslash B \rightarrow S^{1}$ is a fiber bundle, such that $\pi$ agrees with the angular coordinate $\theta$ on the normal disk $\mathbb{D}^2$ when restricted to a neighborhood $B \times \mathbb{D}^2$ of $B$. The closure of the fibers in the fibration are called \emph{pages} and $B$ is called the \emph{binding} of the open book decomposition.\smallskip

Alternatively, an \textit{abstract open book decomposition} of a $(2n+1)$-dimensional manifold $M$ is a pair $(W, \varphi)$, where $W$ is a compact $2n$-dimensional manifold with boundary, and $\varphi: W \rightarrow W$ is a diffeomorphism which restricts to the identity on a neighborhood of $\partial W$, such that
\begin{equation*}
M= W_{\varphi}\bigcup \left(\mathbb{D}^2 \times \partial W \right) 
\end{equation*}
Here, we denote
\begin{equation*}
W_{\varphi}=[0,1] \times W \bigm/ (0, z) \sim (1, \varphi(z)),
\end{equation*}
the mapping torus of $\varphi$.\smallskip

The map $\varphi$ is the \textit{monodromy} and the submanifold $W$ is the \textit{page} of the open book decomposition. We denote $M=\textbf{OB}(W,\varphi)$. \smallskip

\begin{definition}[Giroux]\label{def:Girouxform}
An open book decomposition is said to \textit{support a contact structure} $\xi$ on $M$ if it is the kernel of a contact form $\lambda$ satisfying the following:
\begin{enumerate}
\item $\lambda$ restricts to a contact form on the binding and
\item $d\lambda$ is positively symplectic on the pages $W$, and the orientation on $B$ induced by the contact form agrees with the boundary orientation on $B$ induced by the symplectic form on $W$.
\end{enumerate}

The above contact form $\lambda$ is called a \textit{Giroux form}.
\end{definition} 

\subsection*{Symplectic Lefschetz fibrations}
Let $E$ be a compact $2n$-dimensional manifold with corners, whose boundary is the union of a ``horizontal boundary'' $\partial_h E$ and a ``vertical boundary'' $\partial_v E$ meeting in a codimension 2 corner. Let $\omega=d \lambda$ be an exact symplectic form on $E$ such that both pieces of the boundary are convex. Now consider a smooth map $f: E \to \mathbb{D}^2$ with finitely many critical points $\textit{Crit}(f)$, and denote a regular fiber by $F$. We then say that  $f: E \to \mathbb{D}^2$ is an \textit{exact symplectic Lefschetz fibration} if it satisfies the following properties: 
\begin{enumerate}

\item \textit{(Lefschetz-type critical points)}\\
The critical points of $f$ are nondegenerate, isolated and belong to the interior of $E$. For any $p \in \textit{Crit}(f)$, there are orientation preserving local complex coordinates $(z_{1},\dots , z_{n})$ about $p$ on $E$ and $f(p)$ on $\mathbb{D}^2$ such that, with respect to these coordinates, $f$ is given by the complex map $$f(z_1, \dots, z_n)= z_1^2 + \cdots + z_n^2.$$  Moreover, there is at most one critical point in each fiber of $f$.\smallskip

\item \textit{(Symplectic fibers)}\\  
Denote by $E_z$ the fiber $f^{-1}(z)$ for any $z \in \mathbb{D}^2$. We require that the restriction of $\omega|_{E_z \setminus \textit{Crit}(f)}$ is symplectic, and that the boundary of all fibers is convex. \smallskip

\item \textit{(Conditions on the boundary)}\\
We require that $$\partial_v E= f^{-1}(\partial \mathbb{D}^2), \ \ \partial_h E= \bigcup_{z \in \mathbb{D}^2} \partial (f^{-1}(z)),$$ and $f|_{\partial_v E}: \partial_v E \to \partial \mathbb{D}^2$ and $f|_{\partial_h E}: \partial_h E \to \mathbb{D}^2$ are smooth fiber bundles. Moreover, there is a neighborhood $N(\partial_h E)$ of $\partial_h E$ such that the restriction map $f|_{N(\partial_h E)} : N(\partial_h E) \to  \mathbb{D}^2$ is a product fibration $\mathbb{D}^2 \times N(\partial F)$ where $N(\partial F)$ denotes a neighborhood of $\partial F$. 
\end{enumerate}

We note that the corners of $E$ can be smoothed to make the resulting manifold into an honest Liouville domain $(W, \omega, \lambda)$. 

\subsection*{Iterated planar contact manifolds}
An important recent result of Giroux and Pardon is the following:
\begin{thm}\cite[Thm.\ 1.10]{GP}
Every Weinstein domain $W$ is obtained, up to deformation equivalence, by attaching critical Weinstein handles to a stabilization $W_0\times \mathbb{D}^2$ of another Weinstein domain $W_0$ (the \emph{fiber}) along a collection of Lagrangian spheres.
\end{thm}
See \cite{GP} for a more precise statement. In other words, every Weinstein domain admits a Weinstein Lefschetz fibration over $\mathbb{D}^2$ with Weinstein fibers. Once we are given an exact symplectic Lefschetz fibration on a Weinstein domain, one can further construct an exact symplectic Lefschetz fibration on the codimension two Weinstein fiber, and iterate this process until the Liouville fiber is $4$-dimensional. This idea motivates the following definitions introduced in \cite{Acu}. 

\begin{definition} \label{IPLF}
An \textit{iterated planar Lefschetz fibration} $f: (W^{2n}, \omega) \to \mathbb{D}^2$ on a $2n$-dimensional Weinstein domain $(W^{2n}, \omega)$ is an exact symplectic Lefschetz fibration satisfying the following properties: 

\begin{enumerate}

\item There exists a sequence of exact symplectic Lefschetz fibrations $f_i:(W^{2i}, \omega_i) \to \mathbb{D}^2$ for $i=2, \dots, n$ with $f=f_n$.

\item The total space $(W^{2i}, \omega_{i})$ of $f_{i}$ is a regular fiber of $f_{i+1}$, for $i=2, \dots, n-1$.

\item $f_2: (W^4, \omega_2) \to \mathbb{D}^2$ is a planar Lefschetz fibration, i.e. the regular fiber of $f_2$ is a genus zero surface with nonempty boundary, which we denote by $W^2$.
\end{enumerate}

\subsection*{Examples: $A_k$-singularities} For $n\geq 2$, the unit disk cotangent bundle $W^{2n}=\mathbb{D}^{*}S^n$ admits an iterated planar Lefschetz fibration where each regular fiber is $\mathbb{D}^{*}S^{n-1}$, and the Lefschetz fibration on $\mathbb{D}^{*}S^2$ is planar with fibers $\mathbb{D}^{*}S^1=[0,1] \times S^1$. More generally, consider the $A_k$-singularity, which is given by $$A_k= \{(z_1, \dots, z_n)\in \mathbb{C}^n \mid z_1^2+\dots+z_{n-1}^2+z_{n}^{k+1}=1\} \subset (\mathbb{C}^n, \omega_{std})$$ for $n\geq 3$ and $k\geq2$ (see \cite{Mil} for a more general treatment of complex singularity theory, Milnor fibrations and Brieskorn varieties, and \cite{KvK} for connections to contact topology). It is important to note that the $A_k$-singularity can be expressed as a plumbing of $k$ copies of $\mathbb{D}^{*}S^{n-1}$. Observe that each regular fiber of the Lefschetz fibration on the $A_k$-singularity, defined by the projection onto the last coordinate $z_n$, is $\mathbb{D}^{*}S^{n-1}$. This observation together with the existence of an iterated planar Lefschetz fibration on $\mathbb{D}^{*}S^{n-1}$ imply that the $A_k$-singularity admits an iterated planar Lefschetz fibration.\smallskip

Here we remark that not every $4$-manifold with nonempty boundary admits a planar Lefschetz fibration over $\mathbb{D}^2$. As a counterexample, consider  $T^2 \times \mathbb{D}^2$. Assume to the contrary that $T^2 \times \mathbb{D}^2$ admits a planar Lefschetz fibration. Then there must be a planar Stein fillable contact structure on the boundary of $T^2 \times \mathbb{D}^2$. Note that $\partial(T^2 \times \mathbb{D}^2) = T^3$, which admits a unique Stein fillable contact structure \cite{Eli3} and is known to be nonplanar \cite{Etn}. Hence, $T^2 \times \mathbb{D}^2$ does not admit such a fibration.\smallskip

If $f: W \rightarrow \mathbb{D}^2$ is an iterated planar Lefschetz fibration, then, after smoothing the corners, the boundary of $W$ inherits an open book decomposition whose pages are diffeomorphic to the regular fibers of $f$. The following definition is motivated by looking at the open book decomposition induced by the boundary restriction of an iterated planar Lefschetz fibration. 
\end{definition}

\begin{definition} \label{IPOB}
An \textit{iterated planar open book decomposition} of a contact manifold $(M^{2n+1}, \xi)$ is an open book decomposition $\textbf{OB}(W, \varphi)$ whose page $W$ admits an iterated planar Lefschetz fibration.
\end{definition}

In what follows, we make use of the higher-dimensional correspondence due to Giroux \cite{Gir} and iterated planar open book decompositions to define iterated planar contact manifolds.

\begin{definition} \label{IPCM}
For any $n>1$, an \textit{iterated planar contact manifold} $(M, \xi)$ is a $(2n+1)$-dimensional contact manifold supported by an open book decomposition whose Weinstein page admits an iterated planar Lefschetz fibration.
\end{definition}

\subsection*{Examples: Simply-connected contact $5$-manifolds.} We recall the classification of simply-connected $5$-manifolds due to Barden \cite{Barden}. Roughly speaking, a simply-connected $5$-manifold is determined up to diffeomorphism by its second homology group and its second Stiefel-Whitney class. Every such manifold decomposes uniquely into a connected
sum of prime manifolds $M_k$, with $1\leq k \leq \infty$ a prime power whenever finite, with possibly one further summand $X_j$
with $j = -1$ or $1 \leq j \leq \infty$. The manifold $M_k$ is spin and has $H_2(M_k)\cong \mathbb{Z}_k\oplus \mathbb{Z}_k$, for $1< k < \infty$, and moreover $M_1 \cong S^5$, $M_\infty \cong S^2 \times S^3$. $M_1$ only appears in its own prime decomposition. $X_{-1}$ is the \emph{Wu manifold}, with $H_2(X_{-1})\cong \mathbb{Z}_2$; $X_j$ is non-spin and satisfies $H_2(X_j)\cong \mathbb{Z}_{2^j}\oplus \mathbb{Z}_{2^j}$ for $1\leq j <\infty$; and $X_\infty \cong S^2\widetilde{\times}S^3$, the non-trivial $S^3$-bundle over $S^2$, with $H_2(X_\infty)\cong \mathbb{Z}$.\smallskip

Geiges shows in \cite{Geiges} that simply-connected $5$-manifolds admit contact structures whenever they admit \emph{almost} contact structures (i.e.\ a reduction of the tangent bundle to $U(2)\times \mathbf{1}$, which is a purely topological condition), and that this is equivalent to the vanishing of the third Stiefel-Whitney class. Every $M_k$ carries an almost contact structure, and among the $X_j$'s, only $X_\infty$ does. In \cite{Otto}, van Koert provides supporting open book decompositions for contact structures realizing all homotopy classes of almost contact structures, for each simply-connected $5$-manifold. The monodromy of all the open books corresponding to $M_\infty$ and $X_\infty$ are trivial, and their bindings consist of Lens spaces $L(p,1)$, with $p>3$, and various contact structures, which are all strongly filled by the pages and hence tight. When combined with \cite[Thm.\ 3.3]{SS}, and \cite[Thm.\ 1]{Wen}, one obtains:

\smallskip

\emph{Every almost contact structure in $M_\infty$, $X_\infty$ and $M_1$, as well as in connected sums of these, is achieved by an iterated planar contact structure.}

\smallskip

Among these, we have both the standard tight and overtwisted contact structures on $M_1\cong S^5$, as well as the standard Stein-fillable contact structure on $M_\infty\cong ST^*S^3$. We remark that the open books constructed by van Koert for contact structures in $M_k$, for $1<k<\infty$, are not iterated planar: The pages have the structure of Lefschetz fibrations with fibers having positive genus.

\medskip

A similar discussion on iterated planarity of simply-connected contact 5-manifolds and further work on iterated planarity of contact manifolds will appear in \cite{AEO}.

\subsection*{Almost complex and stable Hamiltonian structures} An \textit{almost complex structure} $J$ on $W^{2n}$ is a linear isomorphism $J: TW \rightarrow TW$ such that $J^2= -1$, and $(W,J)$ is an \emph{almost complex} manifold. An almost complex structure $J$ on a symplectic manifold $(W, \omega)$ is said to be \textit{compatible} with $\omega$ (or \textit{$\omega$-compatible}) if $\omega(Ju, Jv)=\omega(u, v)$, for any $u, v \in TW$, and $\omega(v, Jv) > 0$ for any nonzero vector $v \in TW$. It is well-known that the set of all almost complex structures on a symplectic manifold $(W, \omega)$ is nonempty and contractible. \smallskip

Let $(F, j)$ be a Riemann surface and $(W, J)$ be an almost complex manifold. A \textit{$J$-holomorphic (or pseudoholomorphic) curve} is then a smooth map $u : (F, j) \rightarrow (W, J)$ that satisfies the Cauchy-Riemann equation $du \circ j = J \circ du$.

We now define a special class of compatible almost complex structures on the \emph{symplectization} $(\R \times M, d(e^s \lambda))$. The symplectization of $M$ inherits a natural splitting of the tangent bundle $$T(\R \times M)=\R\langle \partial_s\rangle \oplus \R\langle R_\lambda\rangle \oplus \xi,$$ where $\partial_s$ is the unit vector in the $\R$-direction.\smallskip

An almost complex structure $J$ on $\R \times M$ is called \emph{adjusted} if
 \begin{enumerate}
 \item $J$ is $\R$-invariant,
\item $J(\partial_s)=R_\lambda$ and $J(-R_\lambda)=\partial_s$,
\item $J(\xi)=\xi$,
\item $J|_{\xi}$ is $d\lambda$-compatible.
 \end{enumerate}
 
A \textit{stable Hamiltonian structure} $\mathcal{H}$ on a $(2n+1)$-dimensional oriented manifold $M$ is a pair $(\Omega, \Lambda)$ consisting of a closed 2-form $\Omega$ and 1-form $\Lambda$ defined on $M$ with the following properties:
\begin{enumerate}
\item $\op{ker}\Omega \subset \op{ker}d\Lambda$
\item $\Lambda \wedge \Omega^{n}>0$
\end{enumerate}

We call condition $(1)$ the \emph{stabilizing} condition, whereas condition $(2)$, the \emph{framing} condition. \smallskip

The Reeb vector field $R$ associated to $\mathcal{H}$ is defined by the equations $\Omega(R,\cdot)=0$, $\Lambda(R)=1$. If $\lambda$ is a contact form, then $(d\lambda,\lambda)$ is a stable Hamiltonian structure, which we say is a  \emph{contact} stable Hamiltonian structure. The \emph{symplectization} of $(M, \mathcal{H})$ is $(\mathbb{R} \times M, \omega_{\phi})$, where $\omega_\phi=d(\phi(s)\Lambda)+\Omega$ for $\phi: \mathbb{R} \to (-\epsilon, \epsilon)$ a strictly increasing function, and $\epsilon>0$ small. We have an analogous notion of almost complex structures adjusted to a given stable Hamiltonian structure. 

\subsection*{Symplectic fillings and pseudoconvexity} Let $(W,J)$ be an almost complex manifold with boundary $M = \partial W$. We say that $(W,J)$ has \textit{strictly pseudoconvex} boundary $(M, \xi)$ if $M$ is a regular level set of a $J$-convex function whose gradient points outward at the boundary, i.e.\ $\xi$ is the subbundle of complex tangencies $TM \cap J(TM) \subset TM$, a contact structure whose conformal symplectic structure tames the restriction of $J$. $M$ is called \textit{weakly pseudoconvex} if $d\lambda(\cdot, J\cdot)|_{\xi} \geq 0$.  \smallskip

Recall the definition of weak domination from the introduction. From \cite{MNW}, pseudoconvexity and weak domination are equivalent notions in all dimensions. $(M^{2n-1},\xi)$ is \textit{strongly fillable} if there exists a weak filling $(W^{2n},\omega)$ such that one can find a Liouville vector field $Z$ in a neighborhood of $M$, which is outwards pointing along $M$. We then call $W$ \textit{strong symplectic filling} (or \textit{convex filling}) of $M$. \smallskip

Let $(W, \omega)$ be a symplectic manifold with more than one boundary component, one of which is $(M, \xi)$. If all the boundary components of $W$ are weakly dominated by $\omega$ then $(W,\omega)$ is called a \textit{weak co-filling} of $(M,\xi)$, and $(M, \xi)$ is called \textit{weakly co-fillable}. Co-fillings of contact manifolds can be turned into honest symplectic fillings if other boundary components can be symplectically capped off.

\section{A handle attachment}\label{Handleattachment}

We consider an \textbf{IP} contact manifold $M_1^{2n-1}=\textbf{OB}(W_1^{2n-2}, \varphi)$ with an \textbf{IP} open book decomposition. We will construct a symplectic cobordism $(C^{2n},\omega_{C^{2n}})$ as described in the introduction. The pictorial reference for all of this construction is given in Figure \ref{thehandle}, which contains most of the relevant information, and which the reader will be encouraged to consult in multiple instances as new objects are introduced. We remark that a similar construction has appeared in \cite{DGZ14}, where a cap is used on the pages, rather than an inductively constructed cobordism.

\begin{figure}[h!]\centering
\includegraphics[width=0.7\linewidth]{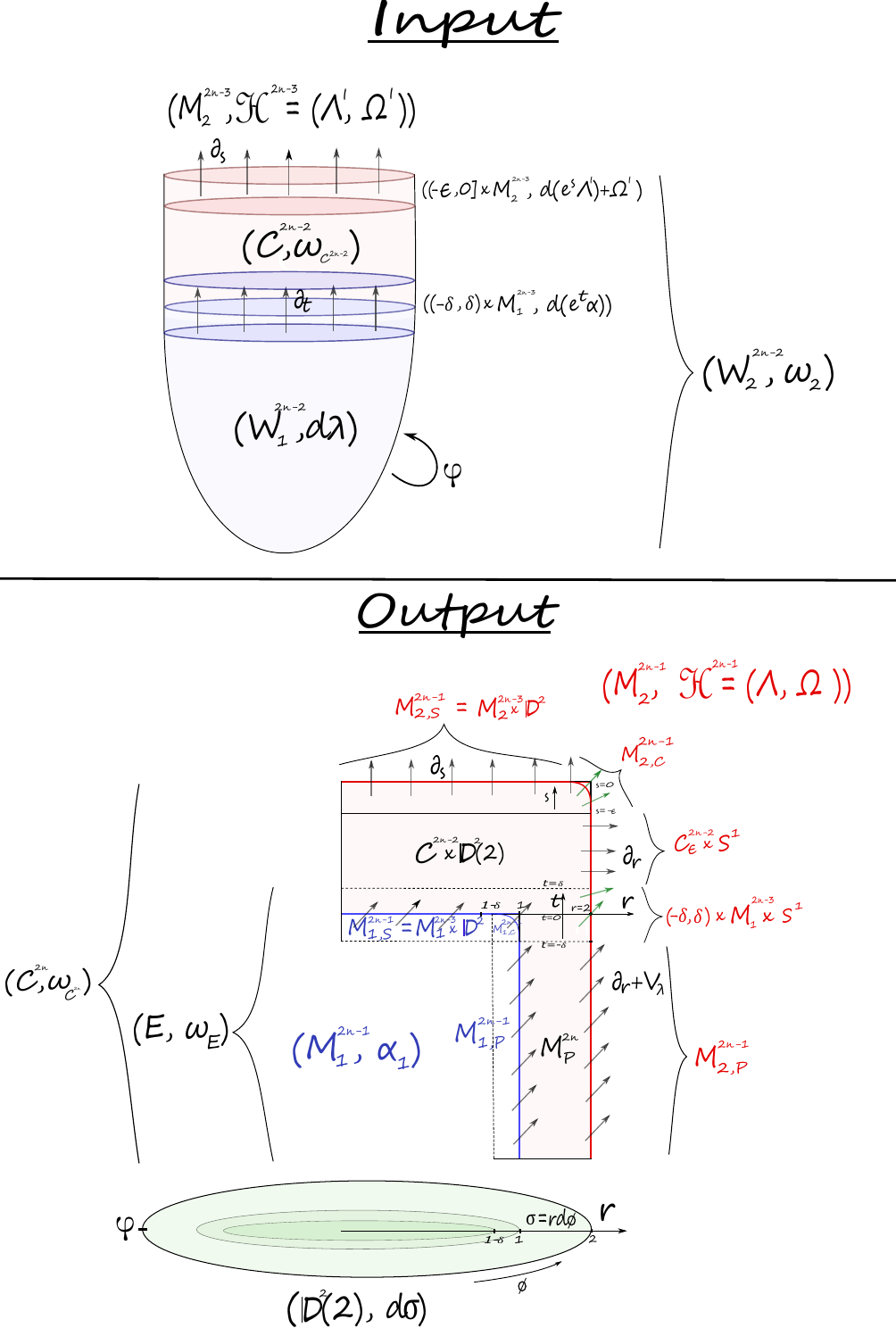}
\caption{\label{thehandle} A diagrammatic picture of the cobordism $C^{2n}$. In order to obtain $C^{2n}$ from the above diagram, one needs to rotate the output manifold in the upper left around the $S^1$-direction of the disk $\mathbb{D}^2(2)$ in the base, and glue it to itself by the action of the monodromy $\varphi$ (as indicated by the symbol $\varphi$ in $\partial \mathbb{D}^2(2)$).}
\end{figure}

\subsection*{Base case \texorpdfstring{$\mathbf{n=2}$}{}} In the case where $M_1^3$ is a planar contact 3-manifold, we make use of a construction originally due to Eliashberg \cite{Eli}, and the further adaptations in \cite{SOBD} used to define the notion of a \emph{spine removal surgery} (cf.\ \cite{Etn2}). \smallskip

By doing $0$-surgery along the binding $M_1^1$ of the planar open book decomposition in $M_1^3=\textbf{OB}(W_1^2,\varphi)$ (with respect to the page framing), we obtain the manifold $M_2^3:=S^1 \times S^2$, the unique $S^2$-fibration over $S^1$. Here, we have used that the pages of the original open book decomposition have genus zero. The resulting $4$-dimensional cobordism $C^4$ carries a symplectic form $\omega_{C^4}$, which satisfies:

\begin{itemize}
\item $\omega_{C^4}=d(e^t\alpha)$ in a collar neighborhood $[0,\delta)\times M_1^3$, where $\alpha$ is a Giroux form for the open book decomposition in $M_1^3$. 
\item $\omega_{C^4}=d(e^s d\theta) + \omega_{S^2}$ in a collar neighborhood $(-\epsilon,0]\times S^1 \times S^2$, where $\epsilon>0$, $s \in (-\epsilon,0]$, and $\omega_{S^2}$ is an area form in $S^2$.
\end{itemize}

In other words, the negative end $M_1^3$ is strongly concave, and the positive end $M_2^3$ is \emph{stable}. The locally defined vector field $\partial_s$ is stabilizing, and $M_2^3$ carries the stable Hamiltonian structure
$$
\mathcal{H}^3:=(i_{\partial_s}\omega\vert_{M_2^3}, \omega\vert_{M_2^3})=(d\theta, \omega_{S^2})
$$

Observe that its Reeb vector field is $\partial_\theta$, and its kernel $\ker d\theta=TS^2$ gives a foliation by spheres.\smallskip

Given a Weinstein manifold $(W_1^{4},\lambda)$ with contact-type boundary $(M_1^3,\alpha)$, we may attach $C^4$ on top of $W_1^{4}$ to obtain a symplectic manifold $$(W_2^{4},\omega_2):=(W_1^4,d\lambda)\bigcup_{M_1^{3}} (C^4,\omega_{C^4})$$ with stable boundary $(M_2^3,\mathcal{H}^3)$. We will use this $4$-dimensional construction as input for our handle construction in dimension $6$.

\subsection*{\textbf{Inductive step}} In arbitrary odd dimension $2n-1\geq 5$, the inductive input is as follows:

\begin{itemize}
\item An \textbf{IP} contact manifold $M_1^{2n-1}=\textbf{OB}(W_1^{2n-2}, \varphi)$, together with an \textbf{IP} open book decomposition.
\item A symplectic cobordism $(C^{2n-2},\omega_{C^{2n-2}})$ with concave contact-type boundary $(M_1^{2n-3},\alpha)=\partial (W_1^{2n-2},d\lambda)$, where $M_1^{2n-3}$ is the binding of the \textbf{IP} open book for $M_1^{2n-1}$, and $\alpha$ is a Giroux form for the open book decomposition induced by the Lefschetz fibration structure in the page $W_1^{2n-2}$; and stable positive boundary $(M_2^{2n-3},\mathcal{H}^{2n-3}=(\Lambda^\prime,\Omega^\prime))$.
\end{itemize}

Observe that $M_1^{2n-3}=\textbf{OB}(W_1^{2n-4},\varphi_1)$, where $W_1^{2n-4}$ is the regular fiber of the Lefschetz fibration in $W_1^{2n-2}$, and $\varphi_1$ is the product of positive Dehn twists along the vanishing cycles in $W_1^{2n-4}$. This means that $M_1^{2n-3}$ is again an \textbf{IP} contact manifold, and so we may assume that the symplectic cobordism $(C^{2n-2},\omega_{C^{2n-2}})$ was constructed in the previous step. \smallskip

We will denote $$(W_2^{2n-2},\omega_2)=(W_1^{2n-2},d\lambda)\cup (C^{2n-2},\omega_{C^{2n-2}}),$$ where the gluing takes place along a collar of the form $((-\delta,\delta)\times M_1^{2n-3},d(e^t\alpha))$, for some $\delta>0$. The result is a symplectic manifold with stable boundary $(M_2^{2n-3},\mathcal{H}^{2n-3})$ (See Figure \ref{thehandle}). The monodromy $\varphi$ extends to $W_2^{2n-2}$ by the identity along $C^{2n-2}$. We will also rename $(W_1^{2n-2},d\lambda)$, replacing it with the slightly enlarged copy $$(W_1^{2n-2},d\lambda) \cup ([0,\delta] \times M_1^{2n-3}, d(e^t\alpha)),$$ without changing notation.\smallskip 

In what follows, we will carry out the construction of the symplectic handle. But first, we introduce some notation.

\subsection*{\textbf{Generalized mapping tori and Giroux forms}} Let $\lambda$ be a Liouville form for a representative of the natural homotopy class of Weinstein structures associated to $W_1^{2n-2}$, such that $\varphi$ is a symplectomorphism with respect to $d\lambda$. We ask also that $\lambda=e^t\alpha$ in the collar neighborhood $(-\delta,0]\times M_1^{2n-3}$, $\delta>0$ as before, where $t$ is the coordinate in the first factor, and $\alpha$ is the given Giroux form in $M_1^{2n-3}$. We denote by $V_\lambda$ the Liouville vector field associated to $d\lambda$, and we assume that $\varphi$ is the identity in $(-\delta,0]\times M_1^{2n-3}$. By a lemma of Giroux (see \cite[Prop.\ 7]{Giroux}), up to isotopy we may assume that $\varphi$ is an \emph{exact} symplectomorphism, i.e.\ we have
$$
\varphi^*\lambda=\lambda - dh,
$$
for some positive smooth function $h:W_1^{2n-2}\rightarrow \mathbb{R}$, which we may take constant equal to $1$ in the $\delta$-collar neighborhood of $M_1^{2n-3}$. We then write 
$$
M_1^{2n-1}=M_1^{2n-3}\times\mathbb{D}^2\bigcup MT(W_1^{2n-2},\varphi),
$$
where  
$$
MT(W_1^{2n-2},\varphi):=W_1^{2n-2}\times \mathbb{R}\Big/(x,0)\sim (\varphi(x),h(x))
$$

We call $MT(W_1^{2n-2},\varphi)$ a \emph{generalized} mapping torus. Observe that, by construction, it carries a well-defined contact form $\alpha_1:=\lambda + d\phi$, where $\phi$ is the coordinate in the $\mathbb{R}$-factor. The following modification will also be useful: For any $r \in \mathbb{R}^+$, we may define 
$$ MT_r(W_1^{2n-2},\varphi)=W_1^{2n-2}\times \mathbb{R}/(x,0)\sim (\varphi(x),h(x)/r).$$ \smallskip

By construction, $MT_r(W_1^{2n-2},\varphi)$ carries a well-defined contact form $\alpha_r:=\lambda + r d\phi$ (and its diffeomorphism type is clearly $r$-independent). This construction is originally due to Giroux, and this $1$-form is the restriction of a \emph{Giroux form} to the generalized mapping torus $MT_r(W_1^{2n-2},\varphi)$. We also have the following model for a Giroux form along $M_1^{2n-3}\times \mathbb{D}^2$: If $(r,\theta)$ are polar coordinates on $\mathbb{D}^2$, we may take $\alpha_1=\sigma + \alpha$. Here, $\alpha$ is the Giroux form on $M_1^{2n-3}$, and $\sigma$ is a Liouville form on $\mathbb{D}^2$, chosen so that $\sigma=rd\theta$ near $r=1$. One needs to smooth the two expressions of $\alpha_1$ so that they glue together smoothly. The complete construction of $\alpha_1$, including how to actually glue the two expressions together, will actually follow from our construction below. Our choice of smoothing of the corner $(M_1^{2n-3}\times \mathbb{D}^2) \cap MT_1(W_1^{2n-2},\varphi)$ will implicitly do the gluing. The result will be a Giroux form $\alpha_1$ on $M_1^{2n-1}$ (so that the induced contact structure is supported by the \textbf{IP} open book decomposition in $M_1^{2n-1}$), which is modelled by the above expressions on each separate piece.\smallskip

We denote $$M_{1,S}^{2n-1}:=M_1^{2n-3}\times \mathbb{D}^2,$$ and call it the \emph{spine} of $M_1^{2n-1}$, and $$M^{2n-1}_{r,P}:=MT_r(W_1^{2n-2},\varphi),$$ and call it the $r$-\emph{paper}. \smallskip

We refer to the $1$-paper $M^{2n-1}_{1,P}$ simply as the \emph{paper} of $M_1^{2n-1}$ (the $2$-paper will be contained in the paper of the resulting manifold $M_2^{2n-1}$. \smallskip

We are now ready to carry out the construction of the symplectic handle.

\subsection*{\textbf{A ``trivial'' symplectic cobordism}} We will now define an open symplectic manifold $(E,\omega_E)$ having $M_1^{2n-1}$ as a contact-type hypersurface. Symplectically, $(E,\omega_E)$ will be symplectomorphic to an open piece of the symplectization of the contact manifold $(M^{2n-1}_1,\alpha_1)$. What follows is an adaptation of a construction in \cite{SOBD} (see also \cite{Mo}). \smallskip

We enlarge the spine $M^{2n-1}_{1,S}$ to $$\widehat{M}^{2n-1}_{1,S}:= M_1^{2n-3}\times \mathbb{D}^2(2),$$ where $\mathbb{D}^2(2)$ denotes the $2$-disk of radius $2$ in $\mathbb{C}$, and we take polar coordinates $re^{i\phi} \in \mathbb{D}^2(2)$. We pick a Liouville form $\sigma$ in $\mathbb{D}^2(2)$, so that $\sigma=rd\phi$ for $r \in (1-\delta, 2]$. We denote the associated Liouville vector field by $V_{\sigma}$, so that $V_{\sigma}=r\partial_r$ in the collar $(1-\delta,2]\times S^1\subset \mathbb{D}^2(2)$. \smallskip

Again for the same $\delta>0$ chosen before, we define $$\mathcal{N}_P^r:= (-\delta,\delta)\times M_1^{2n-3}\times \left(\mathbb{R}\Big/\frac{1}{r}\mathbb{Z}\right)\subset M_{r,P}^{2n-1},$$ which is a collar neighborhood of $\partial M_{r,P}^{2n-1}$. We take coordinates $(t,b,\phi_r)\in \mathcal{N}_P^r$, and we define $\phi:=r\phi_r \in S^1$. We also define $$\mathcal{N}_S:=M_1^{2n-3}\times (1-\delta,2]\times S^{1}\subset \widehat{M}_{1,S}^{2n-1},$$ with coordinates $(b,r,\phi) \in \mathcal{N}_S$, and 
$$ 
M^{2n}_P:= \left\{(r,w): r \in (1-\delta,2], w \in M_{r,P}^{2n-1}\right\}= \bigcup_{r \in (1-\delta,2]} M_{r,P}^{2n-1}.
$$
Moreover, we have a collar 
$$
\mathcal{N}_P:= \left\{(r,w): r \in (1-\delta,2], w \in \mathcal{N}_P^r\right\}= \bigcup_{r \in (1-\delta,2]}\mathcal{N}_P^r\subset M^{2n-1}_P.
$$

Let $E$ be the open manifold
$$
E=(-\delta,\delta]\times \widehat{M}^{2n-1}_{1,S} \bigsqcup M^{2n}_P \Big /\sim \;,
$$
where we identify a tuple $(t,b,r,\phi) \in (-\delta,\delta]\times \mathcal{N}_S \subset (-\delta,\delta]\times \widehat{M}^{2n-1}_{1,S}$, with $(r, t, b, \phi_r=\phi/r) \in \mathcal{N}_P\subset M^{2n}_P$ (see, again, Figure \ref{thehandle}). \smallskip

Observe that $E$ has boundary and corners, and contains a copy of $M_1^{2n-1}$ as a hypersurface (with corners), depicted in blue in Figure \ref{thehandle}. Indeed, we have $$M_1^{2n-1}\cong \{0\}\times M^{2n-1}_{1,S} \bigcup M^{2n-1}_{1,P}\subset E,$$ and its corner is $(\{0\}\times M^{2n-1}_{1,S})\cap M^{2n-1}_{1,P}$. We can always smooth the latter, and we shall make this smoothing explicit below. Moreover, there is a global and well-defined coordinate $r$ on $E$, as well as a coordinate $t$ (which is not globally defined). We denote by $E(t)$ the region of $E$ along which $t$ is defined. We also have a symplectic ``fibration'' $\pi:E\rightarrow \mathbb{D}^2(2)$, whose fibers change topological type. For $r>1-\delta$, the fiber over $(r,\phi)$ is $(W_1^{2n-2},d\lambda)$, and for $r\leq 1-\delta$, it is a copy of the collar neighborhood $((-\delta,\delta]\times M_1^{2n-3}, d(e^{t}\alpha))$. \smallskip

The boundary of $E$ can be written as 
$$
\partial E=\partial_h E \cup  \partial_vE,
$$
where 
$$\partial_h E:=\{t=\delta\}=\{\delta\}\times \widehat{M}^{2n-1}_{1,S} $$
$$\partial_vE:=\{r=2\}= M^{2n-1}_{2,P}$$

The corner of $E$ is then $\partial_h E \cap \partial_v E=\{t=\delta, r=2\}$. \\

We now construct a symplectic form $\omega_E$ in $E$. Observe that we have a well-defined $1$-form $\lambda_E:=\lambda + \sigma$ on $E$, and it is straightforward to check that it is actually Liouville. Then $\omega_E:=d\lambda_E$ is symplectic. Denote by $V_E$ the associated Liouville vector field, which is nothing else than $V_E=V_\lambda + V_\sigma$. Observe that $V_E$ is manifestly positively transverse to the hypersurface $M_1^{2n-1}\subset E$, away from its corner. Along the region $\{r>1-\delta\}\cap E(t)$, which contains the corner of $E$ and that of $M_1^{2n-1}\subset E$, we have $\lambda_E=e^t\alpha+rd\phi$, and so $V_E=\partial_t+r\partial_r$. This means that $V_E$ will be positively transverse to any reasonable smoothing inside $E$ of the corner of $M_1^{2n-1}$. For example, we may choose smoothing functions $F,G:(-\delta,\delta)\rightarrow (-\delta,0]$ satisfying:
$$
\left\lbrace\begin{array}{l} (F(\rho),G(\rho))=(\rho,0), \mbox{ for } \rho\leq-\delta/3\\
G^\prime(\rho)<0, \mbox{ for } \rho>-\delta/3\\
F^\prime(\rho)>0, \mbox{ for } \rho<\delta/3\\
(F(\rho),G(\rho))=(0,-\rho), \mbox{ for } \rho\geq\delta/3
\end{array}\right.
$$

And then we may replace the region $M_1^{2n-1}\cap \{t\in (-\delta,0], r \in (1-\delta,1]\}$, containing the corner of $M_1^{2n-1}$, with the smoothed corner
\begin{equation}
\begin{split}
M_{1,C}^{2n-1}:=&\{(r=1+F(\rho), t=G(\rho),b, \phi_{1+F(\rho)}=\phi/(1+F(\rho))):\\
&\;(\rho,b, \phi) \in (-\delta,\delta)\times M_1^{2n-3}\times S^1 \}
\end{split}
\end{equation}

See Figure \ref{thehandle}. We then obtain a hypersurface of the form 
$$
\widetilde{M}_1^{2n-1}=\widetilde{M}^{2n-1}_{1,S} \cup M^{2n-1}_{1,C} \cup \widetilde{M}_{1,P}^{2n-1},
$$
where $$\widetilde{M}^{2n-1}_{1,S}=M^{2n-1}_{1,S}\backslash \mathcal{N}_S=M_1^{2n-3}\times \mathbb{D}^2(1-\delta),$$ and $$\widetilde{M}_{1,P}^{2n-1}=M_{1,P}^{2n-1}\backslash\mathcal{N}_P^1.$$

Clearly, we have  that $$\widetilde{M}_1^{2n-1} \cong M_1^{2n-1},$$ and so we will drop the tilde from all the notation. Then, $V_E$ is positively transverse to $M_1^{2n-1}$. It follows that $M_1^{2n-1}$ is indeed a contact-type hypersurface, inheriting the contact form $\alpha_1:= \lambda_E\vert_{M_1^{2n-1}}$. Our construction actually implies that $\alpha_1$ is a Giroux form for $M_1^{2n-1}$.  

\subsection*{\textbf{The handle}} We now construct a handle, which we will attach on top of the manifold $E$. Topologically, this handle is of the form $C^{2n-2}\times \mathbb{D}^2(2)$, and we attach it on top of $E$ by the obvious identification (see Figure \ref{thehandle}). Symplectically, we endow it with the $2$-form $\omega_{C^{2n-2}}+d\sigma$, which is indeed symplectic, and by construction this glues smoothly to $\omega_E$. So we get a symplectic cobordism
$$
(C^{2n},\omega_{C^{2n}}):=(E,\omega_E)\cup \left(C^{2n-2}\times \mathbb{D}^2(2),\omega_{C^{2n-2}}+d\sigma\right).
$$

This cobordism contains a collar $$((-\epsilon,0] \times M_2^{2n-1} \times \mathbb{D}^2(2), d(e^s\Lambda^\prime)+\Omega^\prime+ d\sigma),$$ and again has a corner, corresponding to $\{s=0\}\cap\{r=2\}$ in this collar, where $s \in (-\epsilon,0]$. We smooth this corner similarly as before, where we replace $\delta$ by $\epsilon$, and $\rho$ by $\tau \in (-\epsilon,\epsilon)$ in the choice of smoothing functions $F,G$, and we set $r=2+F(\tau)$, $s=G(\tau)$. We rename $C^{2n}$, replacing it with its smoothed version, and where we also remove the region of $E$ which is identified with the negative symplectization of the hypersurface $M_1^{2n-2}$ (this version of $C^{2n}$ is depicted in red in Figure \ref{thehandle}). Then $C^{2n}$ now has $M_1^{2n-2}$ as a concave contact-type boundary component, and its positive boundary component can be written as
$$
M_2^{2n-1}=M^{2n-1}_{2,S} \cup M^{2n-1}_{2,C} \cup \widehat{M}_{2,P}^{2n-1}.
$$
Here we have the following:
\begin{enumerate}
\item $M^{2n-1}_{2,S}=M_2^{2n-3} \times \mathbb{D}^2(2-\epsilon)$ is the \emph{spine} of $M_2^{2n-1}$,
\item $M^{2n-1}_{2,C} \cong (-\epsilon,\epsilon)\times M_2^{2n-3} \times S^1$ is its \emph{smoothed corner}, and
\item if $W_{2,\epsilon}^{2n-2}:= W_2^{2n-2}\backslash((-\epsilon,0]\times M_2^{2n-1})$ and $C_\epsilon^{2n-2}:=C^{2n-2}\backslash((-\epsilon,0]\times M_2^{2n-1})$, then $\widehat{M}_{2,P}^{2n-1}=MT_2(W_{2,\epsilon}^{2n-2},\varphi)=M_{2,P}^{2n-1}\cup (C_\epsilon^{2n-2}\times S^1)$ is its \emph{paper}. 
\end{enumerate}

Topologically, this is just an open book decomposition with page $W_2^{2n-2}$ and monodromy $\varphi$. But, since $W_2^{2n-2}$ is not Liouville, $M_2^{2n-1}$ does not carry any obvious contact structure supported by this open book decomposition. However, it does carry a stable Hamiltonian structure, constructed below. We call this decomposition a \emph{topological} open book decomposition, and denote it by $$M_2^{2n-1}=\textbf{TOB}(W_2^{2n-2},\varphi).$$ 

In summary, we have
$$
\partial C^{2n}=-\textbf{OB}(W_1^{2n-2},\varphi) \ \bigsqcup \ \textbf{TOB}(W_2^{2n-2},\varphi).
$$

\subsection*{\textbf{The stable Hamiltonian structure}} We now make the boundary component $M_2^{2n-1}\subset \partial C^{2n}$ a \emph{stable} one. We write down a vector field $X$, which is Liouville near $M_1^{2n-1}$, and stabilizing near $M_2^{2n-1}$, as follows. \smallskip

Choose bump functions
$$
\beta_1:(-\epsilon,0]\rightarrow [0,1],
$$
$$
\beta_0: [0,\delta)\rightarrow [0,1]
$$
satisfying:
\begin{itemize}
\item $\beta_1\equiv 0$ near $s=-\epsilon$, $\beta_1\equiv 1$ near $s=0$.
\item $\beta_0 \equiv 0$ near $t=\delta$, $\beta_0\equiv 1$ near $t=0$.
\end{itemize}

Denote by $C^{2n}(s)$ the region of $C^{2n}$ in which the coordinate $s$ is defined, and similarly denote $C^{2n}(t)$. \smallskip

We then define the vector field $X$ on $C^{2n}$ by
\begin{equation}\label{stabvf}
X=\left\{ \begin{array}{lr} \beta_1(s)\partial_s + V_\sigma, & \mbox{ along } C^{2n}(s)\\

V_\sigma, & \mbox{ along } C^{2n}\backslash (E \cup C^{2n}(s))\\

\beta_0(t)\partial_t + V_\sigma, &\mbox{ along } C^{2n}(t)\\

V_\lambda + r\partial_r, & \mbox{ along } M_P^{2n}\backslash C^{2n}(t) 

\end{array}\right.
\end{equation}

Then $X$ is transverse to $\partial C^{2n}$, negatively at $-M_1^{2n-1}$, along which is Liouville, and positively at $M_2^{2n-1}$. It is straightforward to check that $X$ actually stabilizes $M_2^{2n-1}$, and it is actually Liouville along $M_{2,P}^{2n-1}\subset \widehat{M}_{2,P}^{2n-1}\subset M_2^{2n-1}$. We obtain a stable Hamiltonian structure 
$$
\mathcal{H}^{2n-1}=(\Lambda,\Omega):=(i_X \omega_{C^{2n}}\vert_{M_2^{2n-1}}, \omega_{C^{2n}}\vert_{M_2^{2n-1}}), 
$$
which is contact along $M_{2,P}^{2n-1}$. \smallskip

This finishes the inductive construction.

\begin{remark}\label{contact} 
Let us write explicit expressions for $\mathcal{H}^{2n-1}$.\\

Along $M_{2,P}^{2n-1}$, we have
\begin{equation}\label{localmodel5}
\mathcal{H}^{2n-1}\vert_{M_{2,P}^{2n-1}}=\left(\lambda+2d\phi,d\lambda\right), 
\end{equation}
and we clearly see in this expression that it is contact along this region.\\

Along the collar $(-\delta,\delta)\times M_1^{2n-1}\times S^1 \subset \widehat{M}_{2,P}^{2n-1}$, we obtain

\begin{equation}\label{localmodel4}
\mathcal{H}^{2n-1}\vert_{(-\delta,\delta)\times M_1^{2n-1}\times S^1}=\left(\beta_0(t)e^t\alpha+2d\phi,d(e^t\alpha)\right). 
\end{equation}

Along $C_\epsilon^{2n-2}\times S^1 \subset \widehat{M}_{2,P}^{2n-1}$, we get
\begin{equation}\label{localmodel2}
\mathcal{H}^{2n-1}\vert_{C_\epsilon^{2n-2}\times S^1}=\left(2d\phi, \omega_{C^{2n-2}} \right). 
\end{equation}

Along the smoothed corner $M_{2,C}^{2n-1}$, we have 
\begin{equation}\label{localmodel3}
\begin{split}
\mathcal{H}^{2n-1}\vert_{M_{2,C}^{2n-1}}=&\left(e^{G(\tau)}\beta_1(G(\tau))\Lambda^\prime+(2+F(\tau))d\phi,\right.\\ &\left.\;d(e^{G(\tau)}\Lambda^\prime) + \Omega^\prime + F^\prime(\tau)d\tau\wedge d\phi\right).
\end{split}
\end{equation}

Finally, along $M_{2,S}^{2n-1}$, we see that 
\begin{equation}\label{localmodel}
\mathcal{H}^{2n-1}\vert_{M_{2,S}^{2n-1}}=\left(\Lambda^\prime+\sigma, \Omega^\prime + d\sigma\right).
\end{equation}

Observe that the expression (\ref{localmodel}) is contact along the regions where $\mathcal{H}^{2n-3}=(\Lambda^\prime,\Omega^\prime)$ is contact. For example, assuming $\mathcal{H}^{2n-3}$ is the stable Hamiltonian structure arising from this construction in the previous step, this is the case for $M_{2,P}^{2n-3} \times \mathbb{D}^2(2-\epsilon)\subset M_{2,S}^{2n-1}$.
\end{remark}

\section{Modification for weak fillings}\label{modification}

In this section, we modify the symplectic cobordism $C^{2n}$ so that we may glue it on top of an arbitrary symplectic manifold $(W^{2n},\omega)$ with a weak boundary component $M_1^{2n-1}\subset \partial W^{2n}$, an \textbf{IP} contact manifold. What follows is an adaptation of the construction in \cite[p.\ 63-64]{SOBD}.\smallskip

We assume that we are given a closed $2$-form $\omega$ in $M_1^{2n-1}$, such that $\omega\vert_\xi>0$, where $\xi=\ker \alpha_1$ is the contact structure in $M_1^{2n-1}$ associated to the Giroux form $\alpha_1$, and supported by the \textbf{IP} open book decomposition. The $2$-form $\omega$ takes the role of $\omega\vert_{M_1^{2n-1}}$ in the presence of $(W^{2n},\omega)$.\smallskip

By induction, we can arrange a closed $2$-form $\eta$ in $M_1^{2n-1}$ such that:
\begin{enumerate}
\item $[\eta]=[\omega] \in H_{dR}^2(M_1^{2n-1})$.
\item $\eta$ is a pullback of a closed $2$-form $\eta^\prime$ in the binding $M_1^{2n-3}$ along $M_{1,S}^{2n-1}\cup M_{1,C}^{2n-1}\cong M_1^{2n-3} \times \mathbb{D}^2$, such that, recursively, $\eta^\prime$ also satisfies conditions (1) and (2) for dimension $2n-3$ (replacing $\omega$ by $\omega\vert_{M_1^{2n-3}}$ in (1)).
\end{enumerate}

The second condition follows from contractibility of $\mathbb{D}^2$, and yields, in particular, that $\eta$ is independent of $\phi \in S^1$ along $M_{1,S}^{2n-1}\cup M_{1,C}^{2n-1}$. Observe that, because of dimensional reasons, $\eta^\prime\equiv 0$ in the base case $n=2$, and (2) is automatic. 

\subsection{Cohomological extension} We now wish to extend $\eta$ to a closed $2$-form to $C^{2n}$. We do this, again, by induction. We will show that we can always extend closed $2$-forms in $M_1^{2n-1}$ satisfying condition (2) above, to suitable closed $2$-forms in $C^{2n}$. 

\subsection*{Base case \texorpdfstring{$\mathbf{n=2}$}{}} The $4$-dimensional cobordism $C^4$ is a result of attaching a collection of $2$-handles to the planar contact $3$-manifold $M_1^3$. In particular, we have $H_3(C^4,M_1^3;\mathbb{R})=0$, since there are no $3$-handles, and it follows from the long exact sequence of the pair $(C^4,M_1^3)$ that the map $H_2(M_1^3;\mathbb{R})\rightarrow H_2(C^4;\mathbb{R})$ is injective. By duality (over the \emph{field} $\mathbb{R}$), the map $H^2(C^4;\mathbb{R})\rightarrow H^2(M_1^3 ;\mathbb{R})$ is surjective. In terms of de Rham cohomology, this means that we can always extend closed $2$-forms in $M_1^3$ to closed $2$-forms in $C^4$. Moreover, by topological reasons, i.e.\ contractibility of the interval, we can choose the extension so that it is independent on the coordinates $s$ and $t$ along the collars where each are defined (meaning independence on both coefficients and differentials). This implies that $\eta$ is a $2$-form in $S^1\times S^2$, and, since $$H_{dR}^2(S^1\times S^2)=H_{dR}^2(S^2)\cong\mathbb{R},$$ $\eta$ can be taken to be a constant multiple of $\omega_{S^2}$ near $M_2^3=S^1\times S^2$.

\subsection*{Inductive step} We assume by induction that we can always extend closed $2$-forms on $M_1^{2n-3}$ satisfying condition (2) (in dimension $2n-3$) to closed $2$-forms on $C^{2n-2}$. Now, take $\eta$ a closed $2$-form in $M_1^{2n-1}$ satisfying (2). First, we observe that we can always extend it to the portion $$M_P^{2n}=\bigcup_{r \in [1,2]} M_{r,P}^{2n-1}\cong M_{1,P}^{2n-1}\times [1,2] \subset C^{2n}$$ in the obvious way, i.e.\ via pull-back, so that it is $r$-independent. The rest of $C^{2n}$ is diffeomorphic to $C^{2n-2}\times \mathbb{D}^2(2)$. Identify $\eta$ with a closed $2$-form $\eta^\prime$ in $M_1^{2n-3}$ along $$M_1^{2n-3}\times \mathbb{D}^2(2)\subset C^{2n}$$ which also satisfies (2), and, by the induction hypothesis, choose an extension of $\eta^\prime$ to $C^{2n-2}$ (also called $\eta^\prime$). Again, by topological reasons, we can assume that $\eta^\prime$ is $t$-independent in the collar neighborhood $$(-\delta,\delta)\times M_1^{2n-3} \subset C^{2n-2},$$ and $s$-independent in $$(-\epsilon,0]\times M_2^{2n-3}\subset C^{2n-2}.$$ We then extend $\eta$ to $C^{2n-2}\times \mathbb{D}^2(2)$ via pullback, so that it is $\mathbb{D}^2(2)$-independent. Observe that the two extensions glue together smoothly, and that our assumptions on $\eta^\prime$ of $s$- and $t$-independence are compatible with the construction of $\eta$ (when this one takes over the role of $\eta^\prime$ in the next step). This finishes the induction argument.

\begin{remark}\label{rindep}
Observe that it follows from the construction that $\eta$ is $r$-, $s$-, and $t$-independent wherever these variables are defined and $\phi$-independent along the region $C^{2n}\backslash M_P^{2n}$.
\end{remark}

\subsection{\texorpdfstring{$\mathbf{\eta}$-perturbation}{}} Let us now fix a closed $2$-form $\eta$ in $C^{2n}$, constructed inductively as above. We modify the symplectic form $\omega_{C^{2n}}$ to the $2$-form
$$
\omega_{C^{2n}}^\eta:=D\omega_{C^{2n}}+\eta,
$$
for a large constant $D\gg 1$. If $D$ is chosen large enough, this $2$-form is indeed symplectic, since the first summand then dominates, and nondegeneracy is an open condition. Since $[\omega]=[\eta]$ and $\omega\vert_\xi>0$ along $M_1^{2n-1}$, by Lemma 1.10 in \cite{MNW}, we get a symplectic form on $[0,1]\times M_1^{2n-1}$ which coincides with $\omega$ on $\{0\}\times M_1^{2n-1}$ and with $Dd\alpha_1+\eta$ on $\{1\}\times M_1^{2n-1}$, for any $D$ sufficiently large. Therefore, after attaching this symplectic cobordism $[0,1]\times M_1^{2n-1}$, we may glue the symplectic manifold $(C^{2n},\omega_{C^{2n}}^\eta)$ to any symplectic manifold $(W^{2n},\omega)$ having $M_1^{2n-1}$ as a weak boundary component. We shall denote $$\mathcal{H}_\eta^{2n-1}:=(i_X\omega_{C^{2n}}^\eta\vert_{M_2^{2n-1}}, \omega_{C^{2n}}^\eta\vert_{M_2^{2n-1}})= D\mathcal{H}^{2n-1}+(i_X \eta\vert_{M_2^{2n-1}}, \eta\vert_{M_2^{2n-1}}),$$ where $X$ is the stabilizing vector field defined in (\ref{stabvf}). Observe that since $\mathcal{H}_\eta^{2n-1}$ is a perturbation of the SHS $\mathcal{H}^{2n-1}$, the framing condition (which is open) still holds. Whereas the stability condition is not open, we will show in next section, nevertheless, that $\mathcal{H}_\eta^{2n-1}$ is indeed still a stable Hamiltonian structure along a specified region of $M_2^{2n-1}$, which is full of $2$-spheres. Observe also that the regions along which $\mathcal{H}^{2n-1}$ is contact, become weakly dominated under the $\eta$-perturbation (since weak domination is an open condition).

\section{Moduli space of spheres}\label{moduli} In this section, we construct an almost complex structure in the symplectic cobordism $C^{2n}$ constructed in Section \ref{Handleattachment}, together with a moduli space of holomorphic spheres.

\subsection{Almost complex structure on a local model} After running the inductive step of the construction of $C^{2n}$, with the particular base case we chose, from Equation (\ref{localmodel}), we obtain the following local model for the stable Hamiltonian manifold $(M_{2}^{2n-1},\mathcal{H}^{2n-1})$ around the binding $M_2^3=S^1\times S^2$ of $M_2^5\subset M_2^7 \subset \dots \subset M_2^{2n-1}$:
\begin{equation}\label{localmodelY}
\left(Y:=S^1\times S^2 \times \mathbb{D}_1^2\times \dots \times \mathbb{D}_{n-2}^2, \left(d\theta+\sum_{i=1}^{n-2} \sigma_i, \omega_{S^2} +\sum_{i=1}^{n-2} d\sigma_i\right)\right)
\end{equation}
Here, $\mathbb{D}_i^2$ is a copy of the $2$-disk $\mathbb{D}^2$, $\omega_{S^2}$ is an area form on $S^2$, and $\sigma_i$ is a Liouville form in the $i$-th disk $\mathbb{D}_i^2$. In particular, its kernel is
$$
\xi_Y:=TS^2\oplus \xi_0,
$$
where $\xi_0=\ker\left(d\theta+\sum_{i=1}^{n-2} \sigma_i\right)$ is a contactization contact structure associated to the Liouville domain $\left(\mathbb{D}:=\prod_{i=1}^{n-2}\mathbb{D}_i^2,\sigma:=\sum_{i=1}^{n-2} \sigma_i\right)$, and so it is a ``confoliation''. Observe also that its Reeb vector field is $\partial_\theta$.

\vspace{0.5cm}

Let us then choose almost complex structures $j_{S^2}$ and $J_{\xi_0}$ which are respectively compatible with $\omega_{S^2}$ and $d\sigma\vert_{\xi_0}$, and define an almost complex structure $J$ 
on $\xi_Y$ by 
$$
J=j_{S^2} \oplus J_{\xi_0}
$$
We extend $J$ to $(-\epsilon,0]\times Y$ so that it maps $\partial_s$ to $\partial_\theta$. Then, $J$ is adapted to the symplectization of the stable Hamiltonian structure $\mathcal{H}^{2n-1}\vert_Y$ along $Y$, which is just the collar $(-\epsilon,0]\times Y$ inside $C^{2n}$. Observe that, whereas $Y$ is weakly pseudoconvex for this choice of $J$, the region $$\{0\}\times S^1\times \{pt\}\times  \mathbb{D}_1^2\times \dots \times \mathbb{D}_{n-2}^2\subset (-\epsilon,0]\times Y$$ becomes a strictly pseudoconvex portion of the boundary. \smallskip

By our specific choice of extension $\eta$, the region $Y\subset M_2^{2n-1}$ remains stable under the $\eta$-perturbation. The analogous local model for the $\eta$-perturbed stable Hamiltonian structure along $Y$, for inductively constructed $\eta$, is
$$
\left(Y=S^1\times S^2 \times \mathbb{D}_1^2\times \dots \times \mathbb{D}_{n-2}^2,D \left(d\theta+\sum_{i=1}^{n-2} \sigma_i, \omega_{S^2} +\sum_{i=1}^{n-2} d\sigma_i\right)+\left(0, \eta\right)\right),
$$
where $\eta=K\omega_{S^2}$ for some $K\in \mathbb{R}$, and $D$ needs to be chosen sufficiently large so that $D+K>0$. In particular, the same $J$ as for the unperturbed stable Hamiltonian structure $\mathcal{H}^{2n-1}$ is adapted to the symplectization of $\mathcal{H}_\eta^{2n-1}$ along $Y$.

\subsection{Extension to \texorpdfstring{$\mathbf{C^{2n}}$}{}} Choose any extension of $J$ to an $\omega_{C^{2n}}$-compatible almost complex structure on $C^{2n}$, so that $J$ is adjusted to the stable Hamiltonian structure $\mathcal{H}^{2n-1}$ near $M_2^{2n-1}$, and makes the latter a weakly pseudoconvex boundary component, which is strictly pseudoconvex along the regions where it is weakly dominated. In the non-perturbed cobordism $(C^{2n},\omega_{C^{2n}})$, we take $J$ to be cylindrical near the concave contact-type boundary component $-M_1^{2n-1}\subset \partial C^{2n}$, and also cylindrical near the convex contact-type portion $$M_{2,P}^{2n-1}\subset M_2^{2n-1} \subset \partial C^{2n}$$ so that in particular the latter region is actually strictly pseudoconvex. In the $\eta$-perturbed one $(C^{2n},\omega_{C^{2n}}^\eta)$, we take $J=J_\eta$ to be a perturbation of the $J$ for the unperturbed data, which is still adjusted to the perturbed stable Hamiltonian structure in the regions which remain stable under the perturbation, and so that the weak boundary piece $M_{2,P}^{2n-1}$ is still strictly pseudoconvex.\smallskip 

As in previous sections, a more detailed inductive construction is also possible, but it will not be relevant for our purposes.

\subsection{A local moduli space} Recall the definition of the local model $Y$ (Eq.\ (\ref{localmodelY})). For this choice of almost complex structure $J$, the spheres 
$$u_{(s,\theta,z)}=\{s\}\times\{\theta\}\times S^2 \times \{z\} \subset (-\epsilon,0]\times S^1\times S^2 \times \mathbb{D}=(-\epsilon,0]\times Y$$ are clearly $J$-holomorphic in $C^{2n}$. Since their normal bundles are trivial, by the Riemann-Roch formula, their index is $2n-2$. Moreover, they are Fredholm regular. Indeed, since $J$ splits, so does the associated normal linearized Cauchy-Riemann operator, which is the direct sum of standard Cauchy-Riemann operators acting on sections of a trivial line
bundle on the two-sphere. These have index 2 by the Riemann-Roch formula (e.g.\ \cite[Appendix A.4]{HZ}). Their kernels have
real dimension two since there are only constant holomorphic sections. Therefore, they are surjective. \smallskip

We will denote by $\mathcal{M}$ the moduli space of $J$-holomorphic spheres in $C^{2n}$ containing the spheres $u_{(s,\theta,z)}$. Observe that we have shown that $\mathcal{M}$ is a manifold of dimension $(2n-2)$ around the latter. 

\subsection{Local uniqueness} In this section, we prove a local uniqueness lemma, which is extracted from the proof of Theorem 7.1 in \cite{MNW}. It is crucial to obtain the results in this paper, and so deserves a separate statement and proof.

\begin{lemma}[Local Uniqueness Lemma]\label{locuniq} By shrinking $\epsilon$ if necessary, any holomorphic map $u:(S^2,j_{S^2})\rightarrow (C^{2n},J)$ in the moduli space $\mathcal{M}$, which intersects $(-\epsilon,0]\times Y$, is (a reparametrization of) one of the spheres $u_{(s,\theta,z)}$. 
\end{lemma}

\begin{proof}
First, we assume that $u$ intersects $\{0\}\times Y$. In this case, we define the (a priori disconnected) open set $$\mathcal{U}:=u^{-1}((-\epsilon,0]\times Y).$$ Then $u\vert_{\mathcal{U}}=(u_1,u_2)$, where $u_1:\mathcal{U}\rightarrow S^2$, and $u_2:\mathcal{U}\rightarrow (-\epsilon,0]\times S^1\times \mathbb{D}$ are holomorphic maps. Since $u_2$ touches the strictly pseudoconvex boundary of $(-\epsilon,0]\times S^1\times \mathbb{D}$ tangentially, it follows that it is constant wherever it is defined. Observing that the only way of escaping the region $(-\epsilon,0]\times Y$ is through the interval direction, it follows that $\mathcal{U}=S^2$. Since $u$ is by assumption an element of $\mathcal{M}$, $u_1:S^2\rightarrow S^2$ is simply covered, and so degree $1$. It follows that $u$ is a reparametrization of a sphere $u_{(0,\theta,z)}$. \smallskip 

In the general case, we will assume to the contrary that we have a sequence $u_k$ of holomorphic spheres intersecting $\{-\epsilon_k\}\times Y$, for $\epsilon_k\rightarrow 0$, and that are not reparametrizations of any of the $u_{(s,\theta,z)}$. Then, a subsequence converges to a nodal configuration $u_\infty$ in $\overline{\mathcal{M}}$ intersecting $\{0\}\times Y$. Then, at least one of its sphere components is a reparametrization of a sphere $u_{(0,\theta,z)}$. All the adjacent sphere components to this particular one intersect $u_{(0,\theta,z)}$, and therefore they intersect $\{0\}\times Y,$ which completely contains the image of $u_{(0,\theta,z)}$. So, it inductively follows that \emph{all} sphere components of $u_\infty$ are reparametrizations of the same $u_{(0,\theta,z)}$ (since different such spheres never intersect). Since all $u_k$ are elements in $\mathcal{M}$, they all have the same homology class, and therefore so does $u_\infty$. It follows that $u_\infty$ is a reparametrization of $u_{(0,\theta,z)}$. But Fredholm regularity implies that every sphere in $\mathcal{M}$ near $u_{(0,\theta,z)}$ is of the form $u_{(s^\prime,\theta^\prime,z^\prime)}$ for some $(s^\prime,\theta^\prime,z^\prime)$, which is a contradiction. This proves the lemma. 
\end{proof}

\begin{remark}
In the proofs below, we shall apply the above Local Uniqueness Lemma to moduli spaces which are obtained via abstract perturbations (in the polyfold sense of Hofer--Wysocki--Zehnder, as discussed below) of moduli spaces which extend $\mathcal{M}$. The precise moduli space depends on the context, and will be specified when used. In our setup, we will only need to introduce abstract perturbations which do not affect the elements of $\mathcal{M}$, and therefore the above lemma continues to hold for the perturbed moduli spaces, with no change in the statements. 
\end{remark}

\section{Weak semi-fillability}\label{semifillproof}

We now proceed to the proof of Theorem \ref{mainthm}.

\begin{proof}
Assume that $(W^{2n},\omega)$ is a symplectic co-filling, having the \textbf{IP} contact manifold $(M^{2n-1},\xi)$ as a boundary component, where $\xi$ is supported by an \textbf{IP} open book decomposition, and other nonempty boundary components. Given any Giroux form $\alpha$ for $\xi$, in the strong case, we may assume that $\omega\vert_{(-\delta,0]\times M}=d(e^t\alpha)$ in some $\delta$-collar neighborhood of $M$. In the weak case, we may take $\omega\vert_{(-\delta,0]\times M}=Dd(e^t\alpha)+\omega\vert_{\xi}$, for some $D>0$ constant. \smallskip

Using the modification of Section \ref{modification}, if necessary, we can then attach the symplectic cobordism $(C,\omega_{C})$ to $(W,\omega)$ along $(M,\xi=\ker \alpha)$, where $\alpha$ is the explicit Giroux form arising from the construction of $(C,\omega_{C})$. We obtain a new boundary component $M^\prime$, which is stable in the strong case, and, in the weak case, has portions which are stable, and others which are weakly dominated (arising as a small perturbation of a contact-type data). Near $M^\prime$, we have a moduli space of $J$-holomorphic spheres for the almost complex structure $J$ from Section \ref{moduli} defined on $C$, stemming from a subregion of the form $(-\epsilon,0]\times Y\subset C$.\smallskip

We now extend the almost complex structure $J$ on $C$ to $W$, so that it makes all the boundary components different from $M^\prime$ strictly pseudoconvex. We then obtain a moduli space of holomorphic spheres $\mathcal{M}$ in $W$, containing the spheres in $C$. Its virtual dimension is $2n-2$, and it is a manifold near the spheres in $(-\epsilon,0]\times Y$. We add a marked point to the domain of each curve in $\mathcal{M}$, obtaining a moduli space $\mathcal{M}_{\bullet}$ of virtual dimension $2n$, and an evaluation map $ev: \mathcal{M}_{\bullet} \rightarrow W$. \smallskip

Observe that, while the spheres in $\mathcal{M}_\bullet$ are somewhere injective if they lie in $(-\epsilon,0] \times Y$, there could be multiple covers somewhere else, or in components of nodal configurations in the compactification $\overline{\mathcal{M}}_\bullet$. We then appeal to the polyfold technology of \cite{HWZ}, together the regularization of constrained moduli spaces of \cite{Ben}, as follows. We will give a fleshed out version of the same argument used in the proof of Theorem 6.1 in \cite{MNW}. 

We view the Gromov compactification $ \overline{\mathcal{M}}_{\bullet}$ of $\mathcal{M}_{\bullet}$ as sitting inside a \emph{Gromov-Witten polyfold} $\mathcal{B}$ \cite[Section\ 2.2, Definition\ 2.29, Section\ 3.5]{HWZ} consisting of (not necessarily holomorphic) stable nodal configurations of spheres with one marked point and possibly multiple components. It comes with a natural evaluation map $ev:\mathcal{B}\rightarrow W$, which extends the one of $\overline{\mathcal{M}}_{\bullet}$. We view the nonlinear Cauchy-Riemann operator $\overline{\partial}_J$ as a \emph{Fredholm section} \cite[Definition\ 4.1]{HWZIII} of a \emph{strong polyfold bundle} $\mathcal{E}\rightarrow \mathcal{B}$ \cite[Definition\ 2.37,2.38, Section 3.6]{HWZ} with zero set $\overline{\partial}_J^{-1}(0)=\overline{\mathcal{M}}_{\bullet}$. We pick a properly embedded smooth path $l:[0,1]\rightarrow W$ with $l(0)=(0,\theta,p,z)\in \{0\}\times Y$, for some $p \in S^2$, so that it intersects each sphere $u_{(s,\theta,z)}$, with $s\in (-\epsilon,0]$, precisely once and transversely, and with $l(1)$ lying in a boundary component $M^{\prime \prime}$ of $W$ different from $M^\prime$, meeting both boundary components transversely inside $W$. Consider the constrained moduli space $\overline{\mathcal{M}}_{\bullet,l}:=ev^{-1}(l)$, as well as $\mathcal{B}^\prime:=ev^{-1}(l)\subset \mathcal{B}$. By \cite[Corollary\ 6.7, Corollary\ 6.8]{Ben}, which are the ep-groupoid versions respectively of \cite[Theorem\ 1.3, Theorem\ 1.5]{Ben} taking isotropy into consideration, together with \cite[Remark\ 1.7]{Ben} to acommodate for the fact that $W$ and $l$ have non-empty boundary, we obtain that $\mathcal{B}^\prime$ is a polyfold with $\partial \mathcal{B}^\prime=ev^{-1}(l(0))\sqcup ev^{-1}(l(1))$, and $\overline{\partial}_J\vert_{\mathcal{B}^\prime}$ is a Fredholm section with index $1$. See also \cite[Section\ 5.1]{Ben} where all the technical assumptions of the theorems just used are explicitly checked.

Strict pseudoconvexity at $M^{\prime\prime}$ implies that $\overline{\partial}_J^{-1}(0)\cap ev^{-1}(l(1))=\emptyset$. Moreover, by our choice of $l$, and the fact that the moduli space $\{u_{(s,\theta,z)}\}$ consisting of the spheres in $(-\epsilon,0]\times Y$ is transversely cut out, the section $\overline{\partial}_J$ is in \emph{good position} \cite[Definition\ 4.12]{HWZIII} (without need to perturb) at the boundary. According to \cite[Theorem\ 4.22]{HWZIII}, we may introduce an abstract perturbation $\mathbf{p}$, which is a multivalued section of $\mathcal{E}$ \cite[Definition\ 3.35, Definition\ 3.43]{HWZIII}, so that:
\begin{itemize}
    \item $\overline{\partial}_J+\mathbf{p}$ is transverse to the zero section, and
    \item The perturbed and constrained moduli space $$\overline{\mathcal{M}}^\prime_{\bullet,l}=(\overline{\partial}_J+\mathbf{p})^{-1}(0)\cap ev^{-1}(l)$$ is a $1$-dimensional compact, oriented, weighted branched orbifold with boundary and corners \cite[Section\ 3.2, Definition\ 3.22]{HWZIII}.
\end{itemize}

Observe that our choice of perturbation implies that the elements of $\overline{\mathcal{M}}^\prime_{\bullet,l}$ approaching $(-\epsilon,0]\times Y$ are still actually $J$-holomorphic curves, and so the Local Uniqueness Lemma \ref{locuniq} continues to hold. Moreover, SFT compactness implies that $(\overline{\partial}_J+\mathbf{p})^{-1}(0)\cap ev^{-1}(l(1))=\emptyset$ for sufficiently small perturbation $\mathbf{p}$. Hence, we obtain $$\partial \overline{\mathcal{M}}^\prime_{\bullet,l}=(\overline{\partial}_J+\mathbf{p})^{-1}(0)\cap ev^{-1}(l(0))=\{u_{(0,\theta,z)}\}$$ by the Local Uniqueness Lemma \ref{locuniq}. But there exist no $1$-dimensional non-empty compact, oriented, weighted branched orbifold with connected boundary, and we get a contradiction. \end{proof}

\section{Nonseparating weak contact-type hypersurfaces}

We now prove Theorem \ref{nonsep}.

\begin{proof}
Assume to the contrary that the \textbf{IP} contact manifold $(M^{2n-1},\xi)$ embeds into a closed symplectic manifold $(W^{2n},\omega)$ as a nonseparating weak contact-type hypersurface. We now adapt the proofs of \cite[Theorem\ 1.15, Theorem\ 2.7]{ABW}, to which we refer the reader for more details. \smallskip

We cut $W$ open along $M$, obtaining a symplectic cobordism $(W_1,\omega_1)$ having two weakly dominated boundary components $M_\pm$, one positive and one negative, and both diffeomorphic to $M$. We ``get rid'' of the negative boundary component $M_-$, by attaching infinitely many copies of $W_1$ at $M_-$, at each step identifying $M_-$ with $M_+$. That is, we inductively define $$(W_n,\omega_n)=(W_{n-1}, \omega_{n-1})\bigcup_{M_-\sim M_+} (W_1,\omega_1).$$ 

After the induction, the result is a noncompact symplectic manifold $(W_\infty, \omega_\infty)$ with boundary $M_+\cong M$, which is positively weakly dominated. By construction, $W_\infty$ contains infinitely many copies of $M$, as weak contact-type hypersurfaces. We attach the perturbed version of the symplectic cobordism $(C^{2n},\omega_{C^{2n}})$ (of Section \ref{modification}) to $(W_\infty,\omega_\infty)$ along its positive boundary $M_+$, obtaining a new symplectic manifold $(W^\prime_\infty, \omega^\prime_\infty)$ with stable boundary $M^\prime$. \smallskip

Observe that, by construction, the symplectic form $\omega_\infty$ is periodic. Therefore, we may choose any $\omega_\infty^\prime$-compatible almost complex structure $J$ in $W_\infty^\prime$, which extends the almost complex structure along $C^{2n}$ of Section \ref{moduli}, which is also periodic along $W_\infty\subset W_\infty^\prime$, and which makes every copy of $M$ inside  $W_\infty$ strictly pseudoconvex. Periodicity then implies that $W^\prime_\infty$ has bounded geometry, which means in particular that closed holomorphic curves with bounded energy have bounded diameter, where the bound on the diameter depends only on the energy bound (see e.g.\ \cite[Lemma 2.5]{Zhang} for a proof of this well-known fact).\smallskip

From Section \ref{moduli}, we obtain a moduli space $\mathcal{M}_\bullet$ of spheres with a marked point in $W^\prime_\infty$, of virtual dimension $2n$, stemming from a region $(-\epsilon,0]\times Y\subset C^{2n}$. We may now use similar arguments as in the proof of Theorem \ref{mainthm}. Namely, we choose a properly embedded path $l:[0,+\infty)\rightarrow W^\prime_\infty$, with $l(0)\in (-\epsilon,0]\times Y$, such that $l$ is transverse to $M^\prime$ and to each sphere $u_{(s,\theta,z)}$. As in the proof of Theorem \ref{mainthm}, by abstractly perturbing the Cauchy-Riemann equation, we obtain a $1$-dimensional noncompact, oriented, weighted branched orbifold $\mathcal{M}^\prime_{\bullet,l}$, consisting of spheres constrained to intersect $l$. Observe that no sequence of spheres in $\mathcal{M}^\prime_{\bullet,l}$ is allowed to traverse the noncompact piece of $W_\infty^\prime$ in such a way that the distance to $M^\prime$ goes to infinity, since otherwise there would exist a sphere in that sequence which tangentially intersects (from below) one of the infinitely many \emph{strictly pseudoconvex} hypersurfaces. This means that there is a sequence of spheres in $\mathcal{M}^\prime_{\bullet,l}$ which intersects $l$ in a sequence of points converging to infinity along the noncompact piece of $W_\infty^\prime$, but always staying at a bounded distance from $M^\prime$. However, this cannot happen because of the diameter bounds on holomorphic spheres imposed by the fact that $W_{\infty}^\prime$ has bounded geometry.\end{proof}

\section{The Weinstein conjecture}

We now prove Theorem \ref{WeinsteinIP}.
\begin{proof}
Let $(M^{2n-1},\xi)$ be an \textbf{IP} contact manifold and $\alpha_0$ be any contact form inducing $\xi$. Let $\alpha_1$ be the Giroux form adapted to the \textbf{IP} open book decomposition with $\ker \alpha_1=\xi$, which we explicitly constructed in Section \ref{Handleattachment}. Then, as in e.g.\ \cite{AH}, after scaling $\alpha_0$ by a constant as necessary, we may construct an exact symplectic cobordism $((-\infty,0]\times M, d\lambda)$, where $\lambda=e^t\alpha_0$ for $t\leq -1$, and $\lambda=e^t\alpha_1$ for $t\in [-\epsilon,0]$ for any given small $\epsilon>0$. \smallskip

We attach the (non-perturbed) symplectic cobordism $(C^{2n},\omega_{C^{2n}})$ constructed in Section \ref{Handleattachment} to $((-\infty,0]\times M, d\lambda)$ along $\{0\}\times M$, obtaining a (non-exact and noncompact) symplectic cobordism $$(W,\omega)=((-\infty,0]\times M, d\lambda) \bigcup_M (C^{2n},\omega_{C^{2n}})$$ with stable boundary $M^\prime$. We extend the almost complex structure $J$ constructed in Section \ref{moduli} to $W$, so that it is cylindrical in the cylindrical ends of $W$. We thus obtain a moduli space $\mathcal{M}$ of spheres in $W$ stemming from $M^\prime$, for which the Local Uniqueness Lemma holds in a subregion of the form $(-\epsilon,0]\times Y \subset (-\epsilon,0]\times M^\prime \subset W.$\smallskip

Choose a properly embedded path $l:[0,+\infty)\rightarrow W$ with $l(0) \in \{0\}\times Y \subset \partial W$, which is transverse to $M^\prime$ and to each sphere $u_{(s,\theta,z)}$, and properly embedded in $W$. Add a marked point to the domain of the spheres, obtaining a moduli space $\mathcal{M}_\bullet$ of virtual dimension $2n$. As in the proof of Theorem \ref{nonsep}, by introducing an abstract perturbation which vanishes near the spheres in $(-\epsilon,0]\times Y$, we obtain a $1$-dimensional non-compact, oriented, weighted branched orbifold $\mathcal{M}^\prime_{\bullet,l}$ by constraining the spheres to lie in $l$, with boundary $ev^{-1}(l(0))$ (by the Local Uniqueness Lemma). By the exactness of $\omega$ in $(-\infty,0]\times M\subset W$, the noncompactness of $\mathcal{M}^\prime_{\bullet,l}$ corresponds only to spheres escaping down the negative end. Observe that, by the exactness of $\omega$ along $(-\infty,0]\times M$, no sphere in $\mathcal{M}^\prime_{\bullet,l}$ can completely lie in $(-\infty,0]\times M \subset W$. It follows that, as we trace $\mathcal{M}^\prime_{\bullet,l}$ along its (possibly multiple) noncompact ends, we obtain a sequence of spheres which intersect $l$ in a sequence of points which diverge along the negative end of $W$. By SFT-compactness \cite{SFT}, this sequence breaks at the negative end, and we obtain a finite energy holomorphic curve in the negative symplectization of $(M,\alpha_0)$. This proves Theorem \ref{WeinsteinIP}.   
\end{proof}

\end{document}